%% file: Article_Arxiv.tex
\documentclass[11pt]{article}

\input{preambule}

\def \HYPR{{\bf (HReg)}}
\def \HYPM{{\bf (HMom)}}
\def \imp{{\Rightarrow}}
\title{Parametrised branching processes:  a functional version\\ of Kesten \& Stigum theorem}
 
\author{C{\'e}cile Mailler\thanks{Department of Mathematical Sciences, University of Bath, Claverton Down, BA2 7AY Bath, UK. c.mailler@bath.ac.uk. CM is grateful to EPSRC for support through the fellowship EP/R022186/1.}~~and Jean-François Marckert\thanks{CNRS, LaBRI ,Universit\'e Bordeaux, 351 cours de la Libération  33405 Talence cedex, France}}

\newcounter{ab}
\definecolor{brit}{rgb}{0.0, 0.26, 0.15}
\definecolor{ame}{rgb}{0.6, 0.4, 0.8}
\definecolor{rouge}{rgb}{.7, .2, .2}
\definecolor{vio}{rgb}{.5, 0, .5}
 \usepackage{stmaryrd}

\def \Q{\mathbb{Q}}
\def \1{\mathbb{I}}

\def \bpar#1{\left\{\begin{array}{#1} }
                      \def \epar { \end{array}\right.}

\newcommand \U[4]{\Upsilon_{#3}^{#4}(#1,#2)}

\usepackage{todonotes}

\counterwithin*{equation}{section}

\begin{document}
\maketitle
\begin{abstract} Let $(Z_n,n\geq 0)$ be a supercritical Galton-Watson process whose offspring distribution $\mu$ has mean $\lambda>1$ and is such that $\int x(\log(x))_+ d\mu(x)<+\infty$. According to the famous Kesten \& Stigum theorem, $(Z_n/\lambda^n)$ converges almost surely, as $n\to+\infty$. The limiting random variable has mean~1, 
and its distribution is characterised as the solution of a fixed point equation. \par
In this paper, we consider a family of Galton-Watson processes $(Z_n(\lambda), n\geq 0)$ defined for~$\lambda$ ranging in an interval $I\subset (1, \infty)$, 
and where we interpret $\lambda$ as the time (when $n$ is the generation). 
The number of children of an individual at time~$\lambda$ is given by $X(\lambda)$, where $(X(\lambda))_{\lambda\in I}$ is a c\`adl\`ag integer-valued process which is assumed to be almost surely non-decreasing and such that $\mathbb E(X(\lambda))=\lambda >1$ for all $\lambda\in I$. This allows us to define $Z_n(\lambda)$ the number of elements in the $n$th generation at time $\lambda$.

Set $W_n(\lambda)= Z_n(\lambda)/\lambda^n$ for all $n\geq 0$ and $\lambda\in I$.
We prove that, under some moment conditions on the process~$X$,
the sequence of processes $(W_n(\lambda), \lambda\in I)_{n\geq 0}$ converges in probability as~$n$ tends to infinity in the space of c\`adl\`ag processes equipped with the Skorokhod topology to a process, which we characterise as the solution of a fixed point equation.

\end{abstract}

\allowdisplaybreaks
\section{Introduction}

The aim of this paper is to discuss some natural models of parameterised branching processes and to introduce a functional version of the Kesten \& Stigum theorem which is one of the prominent results in branching processes theory.

The standard Galton-Watson (GW) process with offspring distribution $\nu=(\nu_j,j\geq 0)$, a probability distribution on $\mathbb{N}:=\{0,1,2,\cdots\}$
is an integer-valued Markov chain $(Z_n, n\geq 0)$ 
such that $Z_0=1$ and, for any $n\geq 0$,
\begin{equation}\label{eq:Zn+1}
Z_{n+1} = \sum_{k=1}^{Z_n} X^{(k,n)}
\end{equation}
where the $(X^{(k,n)},~ k, n\geq 0)$ are i.i.d.\ random variables with distribution $\nu$. 
It is standard to interpret a GW process as describing the evolution of a population structured in generations: for all $n\geq 0$,
$Z_{n+1}$ is seen as the number of individuals in the $(n+1)$-th generation of a population, 
and more precisely, $X^{(k,n)}$ is the number of children of the $k$-th individual of the $n$-th generation. 
From this classical point of view, 
the GW process is the sequence of generation sizes of a the genealogical tree of the branching process (also called the family tree).

\begin{thm}[Kesten \& Stigum \cite{kesten1966}]
Consider $(Z_n, n\geq 0)$ a GW process with offspring distribution~$\nu$, whose mean $\lambda$ is finite. 

If the process is supercritical, that is $\lambda>1$, the sequence $(W_n,n\geq 0)$ defined by
\begin{equation}\label{eq:q83qdsd}W_n = Z_n/\lambda^n,\quad\text{ for all }n\geq 0,
\end{equation}
converges almost surely to a random variable $W$,
and $\P(W>0)>0$ iff $\E(X\log^+(X))<+\infty$ (where $X\sim\nu$). 
Moreover, in this case, $\E W=1$,  and $\P(W=0)=q$ where $q$ is the extinction probability of the branching process~$Z$, that is $q=\P(\exists k~: Z_k=0)$. The value of $q$ is characterised as  the smallest non-negative root of $q= \f(q)$ where $f$ is the probability generating function of $X\sim\nu$~:
  \be
  \f(y)=\E(y^{X}).
  \ee
\end{thm}
The original version of the theorem was written in the multi-type case (see  \Cref{sec:refere} for references and more details). 
In fact the process $(W_n, n \geq 0)$
is a non-negative martingale so that the a.s.\ convergence of $W_n$ to a random variable $W$ is granted. 
Since branching processes can be decomposed at their first generation, 
the limit $W$ is solution to a fixed point equation:
\begin{equation}\label{eq:qdsd}
W \eqd \lambda^{-1}\sum_{j=1}^{Z_1} W^{(j)}\end{equation}
where the $W^{(j)}$ are independent copies of $W$, independent of $Z_1$.
In other words, 
the characteristic function $x\mapsto\Phi(x)= \E(e^{i xW})$ (for $x\in \R$) is solution of
\begin{equation}
\label{eq:qdhtsd}
\Phi(x)= \f(\Phi(x/\lambda)),~~~ \textrm{ for all }x\in \R.
\end{equation}
However, the functional equation~\eref{eq:qdhtsd} (or distribution equation \eref{eq:qdsd}) does not fully characterise $\Phi$, since for any constant $c$, $cW$ is also solution of \eref{eq:qdsd}.
By Seneta~\cite[Th.\ 3.1]{Seneta68}, 
the solution of~\eref{eq:qdhtsd} is unique up to constant factors.
Furthermore, by Kesten-Stigum, $\mathbb E(X(\log X)_+)<\infty\Leftrightarrow \mathbb E(W)=1$, which implies that, in that case, $W$ is the unique solution of~\eref{eq:qdhtsd} with mean~1.

\subsection{A process of GW processes: definition of the model}

We mainly aim at addressing the following question: 
Kesten \& Stigum's theorem is a ``one-dimensional result'', 
since it concerns the limit of the one-dimensional sequence $(W_n, n\geq 0)$. 
There are some natural models in which appears a 
family of GW processes parameterised by a second parameter, 
which can be interpreted as a ``time'' parameter. (We give such an example after the definition.)

\begin{df}\label{def:GWprocess}
Let $I\subset [0, \infty)$, and $(X(\lambda))_{\lambda\in I}$ an almost surely non-decreasing, integer-valued process taking its values in the set of c\`adl\`ag functions $D(I,\R^+)$ equipped with the Skorokhod  topology on all compact subsets of $I$. We define $(Z_n(\lambda))_{\lambda\in I}$ as the process satisfying 
$Z_0(\lambda) = 1$ for all $\lambda\in I$, and, for all $n\geq 0$,
\begin{equation}
\label{eq:Zn+1lambda}
Z_{n+1}(\lambda) 
= \sum_{k=1}^{Z_n(\lambda)} X^{(k,n)}(\lambda)
\end{equation}
where $(X^{(k,n)}, k, n\geq 1)$ is a sequence of i.i.d.\ copies of the process~$X$.
\end{df}
In standard GW processes, $n$ is called the generation: this is the index~$n$ in $Z_n(\lambda)$. We choose to call $\lambda$ the time, and $X$ the offspring process of~$Z$.

\paragraph{A first example:}
Arguably the simplest of these is when the number of children of a node at time $\lambda$ is $X(\lambda)$, where~$X$ is a simple Poisson process on $[0,\infty)$. 
At any given time $\lambda$, the branching process $(Z_n(\lambda),n\geq 0)$ is a simple Galton process whose offspring distribution is ${\sf Poisson}(\lambda)$. For each $n$, $\lambda\to Z_n(\lambda)$ is almost surely non-decreasing, since the number of children of each individual in the tree is non-decreasing as a function of $\lambda$. 
In fact, as $\lambda$ increases, the process of family trees forms a growing family of trees for the inclusion order.
Consider a node of the tree~$u$ which, say, is created at time $t$. 
As a node of the tree at time $t$, its number of children is distributed as $X(t)$, so that the subtree $T^{t}_u$ rooted at $u$ (at time $t$) has the same distribution as a global family tree $T^t$ of a Galton-Watson process with offspring distribution ${\sf Poisson}(t)$.
More examples will be given in \Cref{sec:3examples}.  

\paragraph{A motivation:}

This kind of model arises for example, when one studies the Erd\H os-Renyi graph $G(N,p)$ for $p=\lambda/N$ and $N$ large. 
A vertex $u$ has a ${\sf Binomial}(N-1,p)$ random number of neighbors in the graph, 
approximately ${\sf Poisson}(\lambda)$ distributed when $N$ is large. 
For any fixed~$r>0$, the subgraph of $G(N,p)$ induced by the vertices at graph distance smaller than $r$ to $u$ is well approximated by a Galton-Watson process with offspring distribution Poisson$(\lambda)$ (restricted to its $r$ first generations). 
In many applications (starting from the study of coalescence processes, or as the study of the cluster sizes of $G(N,p)$), $p$ is seen as a varying parameter: 
to each of the $N(N-1)/2$ edges~$e$ of the complete graph $K_N$, assign a weight $w_e$, where the $w_e$ are i.i.d. uniform on $[0,1]$. 
The graph ${\cal G}(N,p)$ obtained by keeping only the edges $e$ of $K_N$ such that $w_e\leq p$ has same distribution as $G(N,p)$, and $p\mapsto {\cal G}(N,p)$ is a graph process which is non-decreasing for the inclusion order. 
Now, if one wants to study the evolution of ${\cal G}(N,\lambda/N)$ in the ball of radius $r$ around a given node, when $N\to +\infty$, for $\lambda \in[a,b]$, then one has to deal with our model: the offspring distribution of the involved nodes, asymptotically, are Poisson process $X=(X(\lambda), \lambda\in[a,b])$.\medbreak

Kesten \& Stigum's theorem implies that the 1-dimensional marginals of the process $(W_n(\lambda), \lambda\in I)_{n\geq 0}$, for parameters $\lambda$ such that $\E(X(\lambda))>1$, converge almost surely. In our model, the $(W_n(\lambda))$ for different values of $\lambda$ are coupled, so that a natural question is: 

Does $(W_n(\lambda), \lambda\in I)$, seen as a sequence of processes indexed by $I$ converges in distribution, or in a stronger sense, as $n\to+\infty?$ \medbreak

Of course, we are also interested in the description of the limit, if it exists.\par

In our main result we assume some properties of the process $X$; we packed these hypothesis into two groups~{\HYPR} and~{\HYPM} ; one concerns ``the regularity of $X$'', and the other ``some moments properties'': \medbreak

\centerline{------------------------------}
{\HYPR}  :  a.s., the process $X$ is c\`adl\`ag on an interval $I\subset(1,+\infty)$; on this interval $X$ is a.s. non decreasing, takes its values in $\mathbb{N}=\{0,1,2,\cdots\}$, and, for any $\lambda \in I$,
\[\E(X(\lambda))=\lambda.\]
\centerline{------------------------------}
Under this hypothesis,   $(Z_n(\lambda), n\geq 0)_{\lambda \in I}$ is a process of GW processes, where for each $\lambda>1$, $(Z_n(\lambda),n\geq 0)$ is a supercritical GW process, whose offspring distribution has mean~$\lambda$. 
Analogously to~\eref{eq:q83qdsd}, we define the process of processes $(W_n(\lambda),\lambda\in I)_{n\geq 1}$ by
\be
W_n(\lambda) = Z_n(\lambda)\,/\,{\lambda^n},\quad \lambda \in I, n \geq 0. 
\ee
\begin{rem}[Comments on {\HYPR} ]
  \bir
  \itr If we remove the non-decreasing property for $X$ then, disappearance of individuals could occur when $\lambda$ grows;  this leads to some complications since the identity of disappearing individuals in their generation matters. We prefer to avoid these complications, even if these models can be defined and investigated.
  \itr Any c\`adl\`ag, non decreasing and non negative  process $Y$ taking its values in $\mathbb{N}$, defined  on an interval $I'$, and satisfying $\E(Y(\lambda))<+\infty$, and $\lambda\to \E(Y(\lambda))$ continuous increasing, is the time-changed of a process $X$ satisfying {\HYPR}. 
  It suffices to set $X(\lambda)=Y(g(\lambda))$ where $g(\lambda)=y$ if $\E(Y(y))=\lambda$ (that is $g$ is the inverse of the map $\lambda\mapsto \E(Y(\lambda))$).
  \eir 
\end{rem} \medbreak

For $\lambda_1<\cdots <\lambda_d \in I$, set
\[\Delta X(\lambda_j):= X(\lambda_j)-X(\lambda_{j-1})\]
with, by convention, $\lambda_0=0$, and $X(\lambda_0)=X(0)=0$.

Denote by  $\Fac{}$ the factorial moments of the increments of $X$ defined as usual by
\be
\Fac{\beta_1,\cdots,\beta_d}(\lambda_1,\cdots,\lambda_d)=\E\l( \prod_{j=1}^d (\Delta X_{\lambda_j})_{(\beta_j)}\r)
\ee
where $(x)_{(r)}= x(x-1)\cdots (x-r+1)$; for example
\[\Fac{3,2,1}(\lambda_1,\lambda_2,\lambda_3)=\E\l[ (\Delta X(\lambda_1))(\Delta X(\lambda_1)-1)(\Delta X(\lambda_1)-2)(\Delta X(\lambda_2))(\Delta X(\lambda_2)-1) (\Delta X(\lambda_3))\r].\]
{{\bf Convention}: we often write $\Fac{\beta_1,\cdots,\beta_d}$ instead of $\Fac{\beta_1,\cdots,\beta_d}(\lambda_1,\ldots,\lambda_d)$; all along the paper, when we need to fix some times, we always choose $(\lambda_1,\ldots,\lambda_d)$ so that this notation is not ambiguous.}

When {\HYPR} holds, $\Fac{1,0}=\Fac{1,0}(\lambda_1,\lambda_2)=\lambda_1$ and  $\Fac{0,1}=\Fac{0,1}(\lambda_1,\lambda_2)=\lambda_2-\lambda_1$.

\begin{lem} Assuming {\HYPR} , the process  $((Z_n(\lambda), \lambda\in I), n\geq 0)$ is (a.s.) well defined, and for each $n$, $\lambda\mapsto Z_n(\lambda)$ is a.s.\ c\`adl\`ag and non-decreasing. 
  \end{lem}
  \begin{proof} \eref{eq:Zn+1lambda} ensures that for each $(\lambda,n)$, $Z_{n+1}(\lambda)$ is a.s. finite, which implies it is well defined. The other properties are clearly inherited from those of $X$.
  \end{proof}
Good control on the increments of $X$ will be needed to control the moments of increments of $W_n$, which is central in our proof of convergence of $W_n$, mainly, in the tightness argument.
\centerline{------------------------------}
{\HYPM} : There exists $\kappa\in (\nicefrac12, 1)$ such that for any $[a,b]\subset I$, there exists a constant $C$ such that, for all $a\leq \lambda_1\leq \lambda_2\leq \lambda_3\leq b$
  \ben
  \Fac{x,y,z}(\lambda_1,\lambda_2,\lambda_3)&\leq& C(\lambda_3-\lambda_1)^{\kappa}, ~~\textrm{for }(x,y,z), 1\leq x+y+z \leq 4, y\geq 1 \textbf{ or } z\geq 1,\label{eq:dsdity}\\
  \Fac{0,y,z}(\lambda_1,\lambda_2,\lambda_3)&\leq& C(\lambda_3-\lambda_1)^{2\kappa}, ~~\textrm{for }(y,z),  1\leq y\leq 2 \textbf{ and } 1\leq z \leq 2.   \label{eq:dsdity2}
  \een
\centerline{------------------------------}
The following lemma, which we prove in Section~\ref{sec:proof_lem_simpl_cond}, 
gives sufficient conditions for \eref{eq:dsdity} and \eref{eq:dsdity2} to hold. 
These are convenient when checking~\eref{eq:dsdity} and~\eref{eq:dsdity2} in practice. 
\begin{lem}\label{lem:easy_HYPM}
{\HYPM} is equivalent to the following condition:  There exists $\kappa\in (\nicefrac12, 1)$ such that for any $[a,b]\subset I$, there exists a constant $C'$ such that for any $a\leq \lambda_1\leq \lambda_2\leq\lambda_3\leq b$,
 \begin{equation}\label{eq:easy1}
 \E\l[(\Delta X(\lambda_2))^2(\Delta X(\lambda_3))^2\r]\leq C'(\lambda_3-\lambda_1)^{2\kappa}\end{equation}
  and
  \begin{equation}\label{eq:easy2}
  \E\Big[(\Delta X(\lambda_3))X(\lambda_3)^3\Big]\leq C'(\lambda_3-\lambda_2)^\kappa.
  \end{equation}
\end{lem}

\subsection{Main results}\label{sec:refs}

For all time $\lambda$, we let $q_\lambda$ denote the non-increasing extinction probability of the process $(Z_n(\lambda),n\geq 0)$. We start by stating the convergence of the finite dimensional distribution (FDD) convergence of the process $W_n(\lambda)$ when $n\to+\infty$: this is an almost sure convergence.
\begin{prop} \label{prp:qsdqs} Assume {\HYPR}. For any $d\geq 1$, for $\lambda_1\leq \cdots \leq \lambda_d$ in $I$, $(W_n(\lambda_i),1\leq i \leq d)$ converges a.s.\ when $n\to+\infty$ to some $d$-tuple of non-negative random variables $(W(\lambda_i),1\leq i \leq d)$. Furthermore, if $\E(X(\lambda_i)\log^{+}(X(\lambda_i)))<+\infty$, then $q_{\lambda_i}=\P(W(\lambda_i)>0)>0$ and $\E(W(\lambda_i))=1$.
\end{prop}
\begin{proof} Kesten \& Stigum's theorem implies the result for each marginal. 
Now, on any probability space on which are defined 
some random variables $\alpha, (\alpha_i, i\geq 0), \beta,(\beta_i, i\geq 0)$, 
if $\alpha_n\as \alpha$ and $\beta_n\as \beta$ then $(\alpha_n,\beta_n)\as (\alpha,\beta)$.
\end{proof}

To state our main result, we use the following convention:
\ben\label{eq:conv} 
\lambda_0&=&X(0)=W(0)=W_n(0)=Z_0(0)=0 
\een
even if we use, in general $X(\lambda)$, $W_n(\lambda)$ and $W(\lambda)$ for $\lambda>1$ elsewhere (notice, for example that $Z_0(\lambda)=1$ for $\lambda \in I$, but we set $Z_0(0)=0$).
These conventions are only used to work more easily with increments (for example, $\Delta W_n(\lambda_i)=W_n(\lambda_{i})-W_n(\lambda_{i-1})=W_n(\lambda_1)$ when $i=1$).

\begin{thm} \label{thm:Geninf} 
If the offspring process $X$ satisfies {\HYPR} and {\HYPM},
then
\[(W_n(\lambda), \lambda \in I) \xrightarrow[n\to+\infty]{proba.} (W(\lambda),\lambda \in I)\]
in $D(I, \R^+)$ equipped with the Skorokhod topology on each compact subsets of $I$, where the process $W$ has a distribution characterised by the following properties:\\
$\bullet$ for any $\lambda\in I$, $\E(W(\lambda))=1$,\\
$\bullet$ its FDD are solution to the following fixed point equation: 
for any $d\geq 1$, any $\lambda_1\leq \cdots \leq \lambda_d$ in $I$,
\ben\label{eq:recGeninf-distr}
\big(W(\lambda_1), \ldots, W(\lambda_d)\big) \eqd \Bigg(\frac1{\lambda_1} \sum_{i=1}^{X(\lambda_1)} W^{\sss (i)}(\lambda_1), \ldots, 
\frac1{\lambda_d} \sum_{i=1}^{X(\lambda_d)} W^{\sss (i)}(\lambda_d)\Bigg),  
\een
where, on the right-hand side, $\big(W^{\sss (i)}(\lambda_1), \ldots, W^{\sss (i)}(\lambda_d)\big)_{i\geq 1}$ are i.i.d.\ copies of $(W(\lambda_1), \ldots, W(\lambda_d))$, independent of $X$.
\end{thm}
\begin{rem}\bir
    \itr We require $I\subset (1,+\infty)$ so that, for each $\lambda \in I$, the GW process $(W_n(\lambda),n\geq 0)$ is supercritical, but the subcritical case can be treated too, but it is trivial. For $\lambda< 1$, $(W_n(\lambda))$ converges a.s. to 0, so that $W_n(\lambda)/\lambda^n$ converges to 0 in $D[0,1)$ on all compact set, and this is true also at $1$, if we exclude the case $\P(X(1)=1)=1$.
   \itr The point of view ``convergence of Fourier transforms'' is discussed in  Section \ref{sec:CFT}.
  \itr The method we use to prove convergence uses a tightness argument: for any $\varepsilon>0$, there exists a compact $K$ of $D([a,b])$ such that $\P(W_n\in K)\geq 1-\varepsilon$ for every $n$. This argument is not strong enough to prove almost sure convergence of $(W_n)$, but it can be conjectured that almost sure convergence holds, possibly under additional regularity assumptions on $X$.
\itr The sufficient condition {\HYPM} comes from our proof strategy using control of moments; it is probably not optimal.
In  \Cref{sec:ECM} we explain that the moments can be exactly computed, 
but because of their complexity, this calculation does not lead to an explicit criterion (which, however would also need fourth moments for $X$, when it is not clear that they are needed).
\itr The process $W$ is in $D(I,\R^+)$, so that it has at most countably-many discontinuities (see Billingsley~\cite[Section 13]{billin}). As $\lambda$ grows, more and more individuals appear in the family tree. When a new node appears,  it appears together with an infinite subtree with positive probability, which provokes a jump of the process~$W$. In fact after leaving~0, the set of jumps of $W$ is dense in~$I$... so that $W$ stops to be continuous as soon as it leaves 0.
\eir
\end{rem}~\medbreak

To establish the convergence in $D(I,\R^+)$ from \Cref{prp:qsdqs}, we mainly need a tightness argument (see Section \ref{sec:tight}), and a lemma to deal with the convergence in probability (\Cref{lem:grfdq}).

\subsection{On the identification of the limiting process}
\label{seq:dqsdq}

As usual when dealing with martingales, 
we know the existence of the limit before knowing anything about it. 
Using a branching property argument similar to the one leading to~\eref{eq:qdhtsd}, it is possible to characterise the limit as the solution of a fixed point equation (as expressed in  \Cref{thm:Geninf} and~\Cref{pro:qdsqhsd}).

Indeed, as in the 1-D case, generation~$n+1$ is formed by the sum of the descendants after $n$ generations of the children of the root: fix $[a,b]\subset I$, $d\geq 1$ and $(\lambda_1,\cdots,\lambda_d)$ such that $a\leq \lambda_1\leq \cdots \leq \lambda_d\leq b$. For all $n$, jointly for $1\leq i \leq d$,
\[Z_{n+1}(\lambda_i) = \sum_{i=1}^{X(\lambda_i)} Z_n^{(i)}(\lambda_i),\]
where $(Z_n^{(i)}\colon n\geq 0)_{i\geq 1}$ is a sequence of i.i.d.\ copies of $(Z_n\colon n\geq 0)$, independent of the offspring process $X$.
This implies that, jointly for $1\leq i\leq d$,
\ben\label{eq:grfdqdrz}
W_{n+1}(\lambda_i) = \frac1{\lambda_i}\sum_{i=1}^{X(\lambda_i)} W_n^{(i)}(\lambda_i),\een
where $(W_n^{(i)}\colon n\geq 0)_{i\geq 1}$ is a sequence of i.i.d.\ copies of $(W_n\colon n\geq 0)$, independent of the offspring process $X$.
Taking the limit as $n\to+\infty$ (since this limit exists by Kesten \& Stigum), we get that the limit satisfies~\eref{eq:recGeninf-distr} (see also Proposition \ref{pro:qdsqhsd}).
For the same reason as in the 1-dimensional case, Equation~\eref{eq:recGeninf-distr} does not characterise the law of {$(W(\lambda_1),\cdots,W(\lambda_d))$}. 
The law of {$(W(\lambda_1),\cdots,W(\lambda_d))$} is characterised as the unique solution of~\eref{eq:recGeninf-distr} having constant mean~1 and finite second moments thanks to the following lemma (proved in \Cref{sec:RP}) in which $\mathcal M_2(1, \ldots, 1)$ denotes 
the set of probability distributions on $[0,\infty)^d$ having mean $(1, \ldots, 1)$ and whose marginals all have finite second moments. 

Note that, under the assumptions of Theorem~\ref{thm:Geninf}, $W$ indeed has constant mean~1, and finite second moment since, by~\cite[Theorem~2.0]{Liu96}, $\mathbb EX(\lambda)^2<+\infty$ implies $\mathbb EW(\lambda)^2<+\infty$. 
\begin{lem}\label{lem:unicite} Assume {\HYPR} and $\mathbb EX(\lambda)^2<+\infty$ for all $\lambda\in I$. 
Let $d\geq 1$ and $1<\lambda_1<\cdots < \lambda_d$ in $I$. We define $\Psi = \Psi_{\lambda_1, \ldots, \lambda_d} \colon \mathcal M_2(1, \ldots, 1) \to \mathcal M_2(1, \ldots, 1)$ as
\[\Psi(\mu) = \mathrm{Law}\left(
\frac1{\lambda_1}\sum_{i=1}^{X(\lambda_1)} U_1^{\sss (i)}, \ldots,
\frac1{\lambda_d} \sum_{i=1}^{X(\lambda_d)} U_d^{\sss (i)}
\right),\]
where the $((U_1^{\sss (i)}, \ldots, U_d^{\sss (i)}))_{i\geq 1}$'s are i.i.d.\ copies of $(U_1,\ldots, U_d)\sim \mu$, independent of the offspring process $X$.
Then $\Psi$ is a contraction for the $L^2$ Wasserstein metric, and in particular, $\Psi$ admits a unique fixed point in $\mathcal M_2(1, \ldots, 1)$.
\end{lem}

\subsubsection{Convergence of  the FDD of $W_n$ with Fourier transforms}
\label{sec:CFT}
We introduce some tools that will play a role in the tightness proof and in some explicit computations that follow. 
For all sequences $(y_i, i \in\cro{a,b})$ indexed by any interval $\cro{a,b}=[a,b]\cap \mathbb{Z}$, the corresponding increment sequence is denoted 
\[\Delta y_i := y_{i}-y_{i-1}, ~~~~~\textrm{for } i \in \cro{a+1,b}.\]
We often write $y\cro{a,b}$ instead of $(y_a,y_{a+1},\cdots,y_b)$.
For all integers $d\geq 1$ and real numbers $\lambda_1<\lambda_2  <\cdots < \lambda_d$ in $I$,
consider the following generating function of the FDD of $X$ and of its increments:
\be
\f_{\lambda\cro{1,d}}( z\cro{1,d}) &:=& \E\l[\prod_{j=1}^d z_i^{X(\lambda_i)}\r],\\
\f^{\Delta}_{\lambda\cro{1,d}}( z\cro{1,d})&:=&
\E\l[\prod_{j=1}^d z_j^{\Delta X(\lambda_j)}\r].
\ee
These generating functions are at least defined on $\overline{B_{\C}(0,1)^d}$ (and of course, each of them can be expressed with the other).
Define the Fourier transform of $(W_n(\lambda_i),1\leq i \leq d)$ by
\ben\label{eq:ftPhin1}
\Phi^{(n)}_{\lambda\cro{1,d}}(x\cro{1,d})&:=&\E\l[\,\exp\l({\rm i}\,\sum_{j=1}^d x_j W_n(\lambda_j)\r)  \,\r].
\een
For any integers $r$ and $d$ such that $1\leq r\leq d$, and any sequence $(x_1,\ldots,x_d)$ define 
\ben\label{eq:dqsdqeh}
\U{x}{\lambda}{r}{d}  = \Big[\underbrace{0,\cdots,0}_{r-1 \textrm{ terms}}, {x_r}/{\lambda_r}, \cdots, {x_d}/{\lambda_d}\Big].\een
The following proposition provides a recursive way to compute $\Phi^{(n)}$.
\begin{prop}\label{pro:qdsqhsd}
For any $\lambda_1 < \lambda_2 <\cdots < \lambda_d$ in $I$, any $x\cro{1,d}\in\R^d$, 
\be
\Phi^{(0)}_{\lambda\cro{1,d}}(x\cro{1,d})= \exp\l[{\rm i} (x_1+\cdots+x_d)\r].
\ee
and for $n \geq 1$, 
\ben\label{eq:recGen}
\Phi^{(n)}_{\lambda\cro{1,d}}(x\cro{1,d})
&=& \f^{\Delta}_{\lambda\cro{1,d}}\l[
\Phi^{(n-1)}_{\lambda\cro{1,d}}(\U x\lambda{1}{d}),\cdots, \Phi^{(n-1)}_{\lambda\cro{1,d}}(\U x\lambda{d}{d})
\r].
\een
Moreover,
$\Phi^{(n)}_{\lambda\cro{1,d}}$ converges pointwise on $\mathbb{R}^d$ to a function $\Phi_{\lambda\cro{1,d}}$ fixed point equation of
\ben\label{eq:recGeninf}
\Phi_{\lambda\cro{1,d}}(x\cro{1,d}) &=& \f^{\Delta}_{\lambda\cro{1,d}}\l[ \Phi_{\lambda\cro{1,d}}(\U x\lambda1d),\cdots, \Phi_{\lambda\cro{1,d}}( \U x{\lambda}dd)\r].  
\een
\end{prop}
As discussed in the beginning of \Cref{seq:dqsdq}, $\Phi_{\lambda\cro{1,d}}$ is moreover the unique solution to \eref{eq:recGeninf} with mean $(1,\cdots,1)$ if $X$ has a finite variance.

\subsection{On explicit computations}
\label{sec:qdqff}

There are two characteristics of the limiting process $W$ that are simple to compute:\\
$\bullet$ the law of $T_{(0,+\infty)}(W):=\inf \{\lambda: W_\lambda >0\}$ the entrance time of $W$  in $(0,+\infty)$ since,
\be
\P\l( T_{(0,+\infty)}(W)>x\r)= q_x,~~~ x\in I
\ee 
the extinction probability of the process $(Z_n(x))_{n\geq 0}$ (the smallest non-negative root of $q=\E(X(x)^{q})$).\\
$\bullet$ the joint moments $\E\l(\prod_{j=1}^m W_{\lambda_j}^{k_j}\r)$ for some fixed $m$, some fixed positive integers $(k_1, \ldots, k_j)$, fixed time $\lambda_1<\cdots<\lambda_m$ for which this quantity exists. 
This is detailed in Section \ref{sec:ECM}: it relies mainly on Lemma~\ref{lem:Mom} which allows to see that there are some polynomial relations between the $\E\l(\prod_{j=1}^m W_{\lambda_j}^{d_j}\r)$ for $d_j\leq k_j$, which can be linearized in all generality. 
Another method consists in extracting the moments using \eref{eq:recGeninf} (we present some of these computations in \eref{eq:rsgf}). \medskip

The computation of the FDD of the processes $W$ proves to be quite technical, since the main tool we have are the Formulae \eref{eq:recGeninf}, which are implicit. In dimension 1,2,3:
\ben
\nonumber \Phi_{\lambda_1} (x_1) &=& \f_{\lambda_1}^{\Delta}\l[\Phi_{\lambda_1}\l(\frac{x_1}{\lambda_1}\r)\r] \\
\label{eq:phi2}\Phi_{\lambda_1,\lambda_2} (x_1,x_2) &=& \f^{\Delta}_{\lambda_1,\lambda_2}\l[\Phi_{\lambda_1,\lambda_2}\l(\frac{x_1}{\lambda_1},\frac{x_2}{\lambda_2}\r),\Phi_{\lambda_1,\lambda_2}\l(0,\frac{x_2}{\lambda_2}\r)\r] \\
\nonumber \Phi_{\lambda[3]} (x_1,x_2,x_3) &=& \f^{\Delta}_{\lambda[3]}\l[\Phi_{\lambda[3]}\l(\frac{x_1}{\lambda_1},\frac{x_2}{\lambda_2},\frac{x_3}{\lambda_3}\r),\Phi_{\lambda[3]}\l(0,\frac{x_2}{\lambda_2},\frac{x_3}{\lambda_3}\r),\Phi_{\lambda[3]}\l(0,0,\frac{x_3}{\lambda_3}\r)\r].
\een
These equations are related, by consistence; the last equation can be rewritten
\[\Phi_{\lambda[3]} (x_1,x_2,x_3) = \f^{\Delta}_{\lambda[3]}\l[\Phi_{\lambda[3]}\l(\frac{x_1}{\lambda_1},\frac{x_2}{\lambda_2},\frac{x_3}{\lambda_3}\r),\Phi_{\lambda_2,\lambda_3}\l(\frac{x_2}{\lambda_2},\frac{x_3}{\lambda_3}\r),\Phi_{\lambda_3}\l(\frac{x_3}{\lambda_3}\r)\r].
\]
The identification of the marginal distributions is difficult too; write
\ben\label{eq:qdsqrzyu} \Phi_\lambda= q_\lambda+(1-q_{\lambda})\Psi_\lambda\een
where $\Psi_{\lambda}$ is the Fourier transform of ${\cal L}(W_\lambda~|~W_\lambda>0)$. Since $q_\lambda$ is known (implicitly, in general), computing $\Psi_\lambda$ is the only real issue: it is solution with mean $1/(1-q_\lambda)$ of
\ben\label{eq:gqdeqf}
q_\lambda+(1-q_{\lambda})\Psi_\lambda(x)= \f_\lambda\l[q_\lambda+(1-q_{\lambda})\Psi_\lambda(x/\lambda)\r],\een
so that
\be
\Psi_\lambda(x)&=&\frac{\f_\lambda\l[q_\lambda+(1-q_{\lambda})\Psi_\lambda(x/\lambda)\r]-q_\lambda}{1-q_\lambda},
\ee
and then, by expanding $f_\lambda(y)=\sum_{m\geq 0}\P(X_\lambda=m)y^m$, we find
\ben\label{eq:rgsgyjj}
\Psi_\lambda(x)= g_\lambda(\Psi_{\lambda}(x/\lambda))
\een
where $g_\lambda$ is the probability generating function of $p(\lambda)=(p_m(\lambda),m\geq 0)$ with
\ben\label{eq:dqsfe}
p_0(\lambda)=0, \textrm{ and, for }j\geq 1,~~ p_j(\lambda)=(1-q_{\lambda})^{j-1}\sum_{m\geq j} \P(X_\lambda=m)\binom{m}j q_\lambda^{m-j}.
\een
In words, $\Psi_{\lambda}$ is the Fourier transform of the limiting martingale of a second Galton-Watson process with offspring distribution $p(\lambda)$, which naturally extincts with probability 0 since $p_0(\lambda)=0$.\medbreak

We will discuss three examples in \Cref{sec:3examples}.

\begin{rem} 
An alternative equation on $\Psi_\lambda$ can be written using the spinal
decomposition of the GW process $(Z_n(\lambda), n\geq 0)$ conditioned on non extinction, 
which $(\widetilde{Z}_n(\lambda), n\geq 0)$ denotes.
Indeed, $\Psi_\lambda$ is the solution of $\Phi^{{\sf Sp}}_\lambda(x)/{(1-q_\lambda)}=\frac{1}{\rm i}\frac{\rm d}{{\rm d}x}\Psi_\lambda(x)$ with $\Psi_{\lambda}(0)=1$, 
where $\Phi^{{\sf Sp}}_\lambda$ is the Fourier transform of the limiting distribution of $\widetilde{Z}_n(\lambda)/\lambda^n$. 
Furthermore, by decomposition at the root, $\Phi^{{\sf Sp}}_\lambda$ is characterised as the solution of $\Phi_\lambda^{{\sf sp}}(x) =  \Phi_\lambda^{{\sf sp}}(x/\lambda) \tilde{f}_\lambda(\Phi(x/\lambda))$ where $ \tilde{f}$ is the probability generating function of $\widetilde{X}-1$, 
where  $\P(\widetilde{X}_\lambda=k)= k\P(X_\lambda=k)/\lambda\; (\forall k\geq 0)$. 
Using these two equations together, we get another another fixed point equation that characterises 
$\Psi_\lambda$. Except possibly in some particular cases, the formula obtained this way is not more convenient than~\eref{eq:gqdeqf}.
\end{rem}

\subsection{Discussion of the related literature}
\label{sec:refere}

Branching processes have been widely studied in probability theory.
They were originally introduced as models for the evolution of populations 
(see, e.g., Haccou \& al.\ \cite{HJV}, and Kimmel \& Axelrod \cite{KA}).
They appear also as combinatorial structures called  trees,
which are one of the simplest models for complex networks. 
Simple families of trees such as uniform planar rooted binary trees with~$n$ internal nodes,
uniform rooted planar trees with~$n$ nodes, 
and uniform rooted labeled trees with~$n$ are equal in distribution to Galton-Watson trees conditioned on having size~$n$. 
Their asymptotic behaviour is thus well-known (Aldous \cite{Aldous}, see also \cite{MM, LGM}). 
We refer also to Devroye~\cite{Devroye} where the theory of branching processes is applied to the analysis of models of random trees such as the binary search tree, Cayley trees, and Catalan trees. 
Branching processes are also a useful tool to study random graphs such as the Erd\H{o}s-Reyni random graph and scale-free random graphs such as the configuration model and the Bar\'abasi \& Albert model (see Bollob\'as \& Riordan~\cite{BR08} for a survey on using branching processes to analyse random graphs).
For a mathematical exposition of some of the existing results on branching processes, 
we refer the reader to, e.g., the surveys of Athreya and Ney \cite{AN}, Asmussen and Hering \cite{AsH}, 
and Lyons and Peres \cite{LP}, in chronological order.

\paragraph{Discussion on Bellman-Harris and Crump-Mode-Jagers processes:}
In this paper, we focus on discrete-time GW processes, meaning that for each $\lambda$,  $(Z_n(\lambda), n\geq 0)$ is a discrete-time Markov chain.
We do not cover the case of continuous-time GW processes (in which each individual has an exponentially-distributed life-time and creates offspring at its death) or their age-dependent generalisations called Bellman-Harris processes (in which the life-time has a non-exponential distributions - see, e.g.~\cite[Chapter~IV]{AN}). 
Another generalisation of continuous-time GW processes are the Crump-Mode-Jagers (CMJ) processes (see, e.g., Jagers~\cite{Jagers}, Nerman~\cite{Nerman}, or Jagers and Nerman~\cite{JN}) in which individuals can create offspring during their whole life-time, for example according to a Poisson process.
Most of these processes exhibit a martingale limit (in the CMJ case, only if the so-called Malthusian parameter exists, see~\cite{Nerman}); 
as far as we know, none of these continuous-time branching processes and their martingale limits have been studied as processes indexed by a parameter as we do here. 

\paragraph{Our model seen as a pruned multi-type Galton-Watson tree:}
It is possible to represent the family trees of our Galton-Watson processes at time $\lambda_1\leq \lambda_2\leq  \cdots\leq \lambda_k$ as pruned multi-type Galton-Watson trees (see, e.g.,\ Athreya \& Ney~\cite[Chapter~V]{AN} for a survey, and Janson \cite{Jan2} for recent limiting theorems).
Indeed, sample the GW tree $(Z_n(\lambda_k), n\geq 0)$,
and, for each node $u$ with offspring process $X_u$,
for all $1\leq i\leq k$, colour in colour~$i$ 
the children of $u$ that appeared when the time parameter belongs to $(\lambda_{i-1}, \lambda_i]$ (set $\lambda_0 = 0$).
The number of the children of~$u$ of colour~$i$ is given by $X_u(\lambda_i)-X_u(\lambda_{i-1})$. 
To get the family tree at time~$\lambda_i$ from this multi-type tree,
we remove all the nodes of color $\geq i+1$. 
However, it is unclear whether the theory of multi-type GW processes could be used to analyse the process $(Z_n(t), n\geq 0)_{t\in I}$.

\paragraph{Discussion on smoothing (or fixed point) equations:}
Fixed point equations analogous to~\eqref{eq:recGeninf-distr} are standard in the theory of branching processes (see, e.g., Liu~\cite{Liu98} and the references therein). 
They are called fixed point or smoothing equations. 
In Lemma~\ref{lem:unicite}, we use the so-called contraction method (see, e.g., R\"osler \& R\"uschendorf~\cite{RR01} for a survey, and Neininger \& Sulzbach~\cite{NS15} where the contraction method is used on functional spaces)
to show uniqueness of the solution with fixed mean and finite variance: 
it is quite straightforward in this case because almost sure convergence 
of $(W_n(\lambda))_{n\geq 1}$ as $n$ tends to infinity is known a priori.

\section{Three examples}

\label{sec:3examples}

\subsection{The binary coupling}
We call this the ``binary'' coupling because, for each time parameter~$\lambda$, 
the GW family tree associated to $(Z_n(\lambda), n\geq 0)$ is binary.
Let $U \sim \uniform[0,1]$ 
and consider the c\`adl\`ag process $(\Bin(\lambda), \lambda\in I_{\Bin})$, 
where
\be
\Bin(\lambda) &:=& 2\, \1_{U \leq \lambda/2}, \textrm{ for any }\lambda\in
I_{\Bin}:=(1,2].
\ee
The process $\Bin$ is a non-decreasing process that is constant in $I_\Bin$, except at the random time $\lambda=2U$ at which  it jumps from~0 to~2.
Moreover, since $\1_{U\leq \lambda/2}\sim \Bernoulli(\lambda/2)$, we have
\[\E \Bin(\lambda)=\lambda.\]

The interval $I_{\Bin}$ is the range of time parameters $\lambda$ for which $\E(\Bin(\lambda))>1$. 
Thus, $(\Bin(\lambda))_{\lambda\in I_\Bin}$ can be used as the offspring process  in Definition~\ref{def:GWprocess}.

To describe the distribution of $(\Bin(\lambda_1),\cdots,\Bin(\lambda_d))$  for $0\leq \lambda_1<\cdots <\lambda_d\leq 2$, set ${\sf FI}=\inf\{k : \Bin(\lambda_k)=2\}$, the first index where $\Bin(\lambda_k)$ is equal to~2; we have
\ben\label{eq:FI}
\P({\sf FI}=k)= (\lambda_k-\lambda_{k-1})/2, \textrm{ for }k\in\cro{1,d}
\een
and $\P({\sf FI}=+\infty)=\P(\Bin(\lambda_d)=0)=1-\lambda_{d}/2$.

 For $\lambda_1\leq \cdots \leq \lambda_m$ elements of $I_{\Bin}^m$,
\be
\f^{\Bin}_{\lambda_1,\cdots,\lambda_m}(x_1,\cdots,x_m)&=&(1-\lambda_m/2)+\sum_{k=1}^m \frac{\lambda_k-\lambda_{k-1}}2\prod_{i=k}^m x_i^2.
\ee
\begin{prop}The process $\Bin$ satisfies {\HYPR} and  {\HYPM}, so that \Cref{thm:Geninf} applies when the offspring process $X=\Bin$, on $I_{\Bin}=(1/2,1]$. 
  \end{prop}
  \begin{proof} It is straightforward to check that {\HYPR}  holds.
  Now, we check {\HYPM}  using \Cref{lem:easy_HYPM}: we have
    $\Fac{0,2,2}{}= 0$ (since either
    $\Delta X(\lambda_2)=0$ or $\Delta X(\lambda_3)=0$), 
    which implies~\eqref{eq:easy1}.
To prove that~\eqref{eq:easy2} holds, write
 $\E_{\Bin}((\Delta X(\lambda_3))X(\lambda_3)^3)= 16\,\P_{\Bin}(X(\lambda_3)=1,X(\lambda_2)=0)=8\Delta \lambda_3$.
 \end{proof}

We now comment on the properties of the limiting process $(W(\lambda), \lambda\in I_{\Bin})$. We first look at the one-dimensional marginal distributions.
From Equation~\eqref{eq:qdhtsd}, we get that, for all $x\in\mathbb R$, $\lambda\in(1,2]$,
\[\Phi_{\lambda}(x)=1-\frac{\lambda}2+\frac{\lambda}2\Phi_{\lambda}(x/\lambda)^2.\]
In this case $q_{\lambda} = (2-\lambda)/\lambda$, and using \eref{eq:rgsgyjj} and \eref{eq:dqsfe} (or \eref{eq:gqdeqf}), $\Psi_\lambda$ is solution to
\be
\Psi_\lambda(x)=  (\lambda-1)\Psi_\lambda \left(  {x}/{\lambda} \right) ^{2}+(2-\lambda)\Psi_\lambda\left( { {x}/{\lambda}} \right),
\ee
so that $p(\lambda)$ is the distribution $(\lambda-1)\delta_2+(2-\lambda)\delta_1$. Hence, $\Psi_\lambda$ is the distribution of 
\[ Z  \eqd B_{\lambda-1} \frac{(Z+Z')}{\lambda}+ (1-B_{\lambda-1} )\frac{Z}{\lambda}=B_{\lambda-1}\frac{Z'}{\lambda}+\frac{Z}{\lambda}\eqd \sum_{j\geq 1} \frac{B_{\lambda-1}^{(j)} Z^{(j)}}{\lambda^j}\]
where $Z,Z',Z^{(1)},Z^{(2)},\cdots$ are i.i.d. copies of $Z$, and the $B_{\lambda-1}^{(j)}$ are i.i.d. Bernoulli random variable with time parameter $\lambda-1$, all these random variables are independent (and the $Z_j$ have mean $1/(1-q_\lambda)$). 
We were not able to find an explicit solution of $\Psi_\lambda$. The first moments of $W_\lambda$ conditioned to be $>0$ are
\[1,{\frac {\lambda /2}{\lambda-1}},{\frac {\lambda /2}{ \left( 
\lambda-1 \right) ^{2}}},{\frac {3\lambda/2}{ \left( \lambda+1
 \right)  \left( \lambda-1 \right) ^{3}}},{\frac {3\lambda\,
 \left( \lambda+5 \right)/2 }{ \left( \lambda+1 \right)  \left( \lambda-
1 \right) ^{4} \left( {\lambda}^{2}+\lambda+1 \right) }},{\frac 
{15\lambda\, \left( 2\,{\lambda}^{2}+3\,\lambda+7 \right) /2}{ \left( {
\lambda}^{2}+1 \right)  \left( {\lambda}^{2}+\lambda+1 \right) 
\left( \lambda+1 \right) ^{2} \left( \lambda-1 \right) ^{5}}}\]
 and $\E(W_\lambda|W_\lambda>0)={\frac {\lambda}{2(\lambda-1)}}$, and its variance is 
$\frac{\lambda(2-\lambda)}{4(\lambda-1)}$ which goes to $\infty$ at 1, and is 0 at 2 (since $W_2=2$ a.s.).\par
We were not able to say something interesting on the $2$-dimensional marginal distributions.

\subsection{The geometric coupling}

Geometric random variables ({with support $\{0,1,2,\cdots\}$}) are important in GW theory because,
conditionally on the total number of nodes $n$, 
the family tree of a GW process with this offspring distribution is uniform 
among the set of (unlabelled) trees with $n$ nodes 
(and this holds for any parameter $\notin\{0,1\}$ of the geometric distribution).

Let $\Ber=(\Ber(\lambda),0\leq \lambda\leq1)$, where $\Ber(\lambda) = 1_{U\leq \lambda}$ and $U\sim {\sf Uniform}[0,1]$. Let $(\Ber^{(i)}, i\geq 0)$ be a sequence of independent copies of $\Ber$.
Define 
\be
\Geom(\lambda)& = &\inf\l\{i\in\{0,1,2,\cdots\}: \Ber^{(i)}\l(\frac{1}{1+\lambda}\r)=1\r\}, \textrm{ for } \lambda \in I_{\Geom}       := (1,+\infty),
\ee
the first success in this sequence of Bernoulli trials; we have $\P(\Geom(\lambda)=k)=\lambda^{k}/(1+\lambda)^{k+1}$, and
\[\E\Geom(\lambda)=\lambda.\] 
The inversion $\lambda\mapsto 1/(1+\lambda)$ allows to get a non-decreasing process, 
while the corresponding processes $\Ber^{(i)}\l(1/(1+\lambda)\r)$ are non-increasing. 
If the sequence $(\lambda_1,\cdots,\lambda_d)$ is non decreasing, 
then the sequence $\Geom(\lambda_1),\cdots,\Geom(\lambda_d)$ is a Markov chain. 
Indeed, given $\Geom(\lambda_{j-1})$, 
\[\Geom(\lambda_{j})
\eqd \Geom(\lambda_{j-1})+\Ber\l[\frac{\lambda_j-\lambda_{j-1}}{1+\lambda_{j}}\r](1+\Geom^\star(\lambda_{j})),\]
where the random variables in the right hand side are all independent, 
so that
\ben\label{eq:geom}\l[\Geom(\lambda_1),\cdots,\Geom(\lambda_d)\r]\eqd 
\Bigg[G_1+\sum_{i=2}^j \Ber^{(i)}\l[\frac{\lambda_j-\lambda_{j-1}}{1+\lambda_{j}}\r](1+G_i),1\leq j \leq d\Bigg]\een
where the $G_j$ are distributed as $\Geom(\lambda_j)$, and all the variables $G_j$ and  $\Ber^{(i)}$ are independent.
For $\lambda_1\leq \cdots \leq \lambda_m$ elements of $I_{\Geom}^m$ 
\be
\f^{\Geom}_{\lambda_1,\cdots,\lambda_m}(x_1,\cdots,x_m)&=& G_{\lambda_1}\l(\prod_{i=1}^m x_i\r) \prod_{j=2}^m \l[\frac{1+\lambda_{j-1}}{1+\lambda_{j}}+ \frac{\lambda_j-\lambda_{j-1}}{1+\lambda_{j}} G_{\lambda_j}\l(\prod_{i=j}^m x_j\r)\prod_{i=j}^m x_j \r],
\ee
where $G_j(x)=1/(1+(1-x)\lambda)$ is the generating function of $\Geom(\lambda)$.
\begin{prop}The process $\Geom$ satisfies {\HYPR} and  {\HYPM}, so that \Cref{thm:Geninf} applies when the offspring process $X=\Geom$, on $I_{\Geom}=(1,+\infty)$. 
  \end{prop}
  \begin{proof} It is straightforward to check that {\HYPR} holds. We check {\HYPM}  using \Cref{lem:easy_HYPM}: 
  first note that
    $\Fac{0,2,2}{}= 4\lambda_2\lambda_3(\Delta \lambda_2)(\Delta \lambda_3)$, which implies~\eqref{eq:easy1}. 
    For~\eqref{eq:easy2}, we write
    $\E_{\Geom}((\Delta X(\lambda_3))X(\lambda_3)^3)= (\Delta \lambda_3)Q$,
    where
    \[Q= 24\,{\lambda_{{3}}}^{3}+18\, \left( \lambda_{{2}}+2 \right) {\lambda_{
{3}}}^{2}+2\, \left( 6\,{\lambda_{{2}}}^{2}+12\,\lambda_{{2}}+7
 \right) \lambda_{{3}}+6\,{\lambda_{{2}}}^{3}+12\,{\lambda_{{2}}}^{2}+
 7\,\lambda_{{2}}+1\]
is a polynomial which is bounded on any compact $[a,b]\subset I$. These moments are computed by some differentiations of $\f^{\Geom}_{\lambda\cro{1,3}}$.
      \end{proof}
      
      We now comment on the properties of the limiting process. In particular, we look at its 1- and 2-dimensional marginals.
By Equation~\eqref{eq:qdhtsd}, we get that, for all $\lambda>1$, for all $x\in\mathbb R$,
\ben\label{eq:fqdqhtrhs}
\Phi_{\lambda}(x)= \big(1+\lambda- \lambda\Phi_{\lambda} ( x/\lambda)\big)^{-1}
\een
and the only solution with mean 1 is
\[\Phi_{\lambda}(x)= (1-\lambda +\mathrm i x)/(1-\lambda+ \mathrm i x\lambda),\]
which we identify as the Fourier transform of the distribution of the random variable
\be
\Ber(p_\lambda)\, {\sf Expo}(p_\lambda) ~~\textrm{ for }p_\lambda=\frac{\lambda-1}\lambda=1-q_{\lambda}
\ee where ${\sf Expo}(a)$ stands for an exponential random variable with parameter $a$ (and mean $1/a$), independent from the Bernoulli random variable. 

Hence, the probability of extinction is $1/\lambda$, and 
${\cal L}(W_\lambda~|~W_\lambda>0)$ is the law of ${\sf Expo}(p_\lambda)$.

The Fourier transform of the second dimensional distribution is a bit involved:
\begin{equation}\f^{\Delta,{\sf Geom}}_{\lambda_1,\lambda_2}( x_1,x_2)={\frac {1+\lambda_1-\lambda_1x_2}{ \left( 1+\lambda_2-\lambda_2x_2 \right)  \left(1+\lambda_1- \lambda_1x_1 \right) }}.
\end{equation}
A solution can be found:
\begin{equation}\label{eq:ggez}\f^{\Delta,{\sf Geom}}_{\lambda_1,\lambda_2}( x_1,x_2) = a+b\Phi_{\lambda_1}(x_1) +c\Phi_{\lambda_2}(x_2)+d\Phi_{\lambda_1}(x_1)\Phi_{\lambda_2}(x_2)
\end{equation}
and this is solution for all $(a,b,c,d)$ such that
\[c=1-d,b=1-d,a=d-1\] and $d$ is a root of the following quadratic polynomial
\begin{multline*}
  -x_1x_2 \left( \lambda_2-1 \right)  \left( \lambda_1-1 \right) {d}^{2}\\
  +{\rm i}d \Big( {\lambda_1}^{2}\lambda_2x_2+\lambda_1{\lambda_2}^{2}x_1-{\lambda_1}^{2}x_2-\lambda_1\lambda_2x_1-{\rm i}{\lambda_2}^{2}-\lambda_2x_1-\lambda_2x_2+2\,{\rm i}\lambda_2+2\,{\rm i}\lambda_1x_1x_2+{\rm i}\lambda_2x_1x_2\\
    +x_1+x_2-2\,{\rm i}\lambda_1\lambda_2+{\rm i}\lambda_1-3\,{\rm i}\lambda_1\lambda_2x_1x_2+{\rm i}\lambda_1{\lambda_2}^{2}-{\rm i} \Big)\\
  -{\rm i}{\lambda_1}^{2}\lambda_2x_2-{\rm i}\lambda_1{\lambda_2}^{2}x_1-2\,\lambda_1\lambda_2x_1x_2+{\rm i}\lambda_1x_1+{\rm i}\lambda_2x_2+{\lambda_1}^{2}\lambda_2+\lambda_1{\lambda_2}^{2}-{\lambda_1}^{2}-2\,\lambda_1\lambda_2-{\lambda_2}^{2}+\lambda_1+\lambda_2
\end{multline*}
This  is not really informative, and we hope that some reader will succeed in finding a more classical representation of this distribution\footnote{To get this formula, we guess the form \eref{eq:ggez}, so that finding $(a,b,c,d)$ such that $\Phi_{\lambda_1,\lambda_2} \left( x_1,x_2 \right) =\f^{\Delta,{\sf Poi}}_{\lambda_1,\lambda_2}\l[\Phi_{\lambda_1,\lambda_2} \left( \frac{x_1}{\lambda_1},\frac{x_2}{\lambda_2} \right),\Phi_{\lambda_1,\lambda_2}\left( 0,\frac{x_2}{\lambda_2} \right)\r]$, and $\Phi_{\lambda_1,\lambda_2} \left( x_1,x_2 \right)$ has the right marginal, is a matter to solve some polynomial equations in $(a,b,c,d)$ and coefficients in the set of rational fractions $\Q[[x_1,x_2,\lambda_1,\lambda_2]]$. This can be done by hand or using a computer algebra system, and the computation of a Kronur basis.}.

\subsection{The Poisson coupling}
Take a standard Poisson process $\Poi:=(\Poi(\lambda), \lambda\in I_{\Poi})$ with intensity~1. 
We have
\[\E\Poi(\lambda)=\lambda, \textrm{ for } \lambda\in I_{\Poi}:=(1,+\infty).\]
For any $\lambda_1\leq \cdots \leq \lambda_m$ elements of $I_{\Poi}$,
\be
\f^{\Poi}_{\lambda_1,\cdots,\lambda_m}(x_1,\cdots,x_m)
&=& \prod_{k=1}^m \exp\Bigg((\lambda_k-\lambda_{k-1})\bigg(-1+\prod_{j=k}^m x_j\bigg)\Bigg), 
\ee
since $(\Poi(\lambda_k)-\Poi(\lambda_{k-1}), 1\leq k \leq m)$ are 
independent Poisson random variables with respective parameters $\lambda_k-\lambda_{k-1}$.
\begin{prop}The process $\Poi$ satisfies {\HYPR} and  {\HYPM}, so that \Cref{thm:Geninf} applies when the offspring process $X=\Poi$, on $I_{\Poi}=(1,+\infty)$. 
  \end{prop}
  \begin{proof} Again, it is straightforward to check that {\HYPR} holds. 
  We now check {\HYPM}  using \Cref{lem:easy_HYPM}: we have
    $\Fac{0,2,2}{}= (\Delta \lambda_2)^2(\Delta \lambda_3)^2$, which implies~\eqref{eq:easy1}. 
    For~\eqref{eq:easy2}, we write
    $   \E_{\Poi}((\Delta X(\lambda_3))X(\lambda_3)^3)= ( {\lambda_{{3}}}^{3}+6\,{\lambda_{{3}}}^{2}+7\,\lambda_{{3}}+1) \Delta \lambda_3$. Since the polynomial in front of $\Delta \lambda_3$ is bounded on any compact $[a,b]\subset I$, this implies~\eqref{eq:easy2}.
These moments can be computed by some differentiations of $\f^{\Poi}_{\lambda\cro{1,3}}$.
      \end{proof}
The extinction probability is solution of $e^{\lambda(q_\lambda-1)}=q_\lambda$, which implies that $q_\lambda$ can be expressed in terms of $\lambda$ using the LambertW function, or reciprocally, given $q=q_\lambda\in(0,1)$, the corresponding $\lambda$ is  $\lambda = \frac{\ln(q)}{q-1}$.
The Fourier transform $\Phi_\lambda$ of $W(\lambda)$ satisfies
$\Phi_\lambda(x)=\exp\l[\lambda(\Phi_\lambda(x/\lambda)-1)\r]$.\par 
Let us turn our attention toward ${\cal L}(W_\lambda~|~W_\lambda>0)$; using \eref{eq:dqsfe} and plugging $\P(X_\lambda=m)=\lambda^m e^{-\lambda}/m!$, we get $p_0(\lambda)=0$ and for $j\geq 1$
\be
p_j(\lambda)&=& \frac{(1-q_{\lambda})^{j-1}}{j!}\lambda^je^{-\lambda}\sum_{m\geq j} \frac{\lambda^{m-j}q_\lambda^{m-j} }{(m-j)!} 
= \frac{\lambda [\lambda(1-q_{\lambda})]^{j-1}}{j!}e^{\lambda(q-1)}.
\ee
Hence $\Psi_\lambda$ is solution to
\[\Psi_\lambda(x)= \frac{-1+q^{\Psi_\lambda(x/\lambda)}}{(-1+q)q^{\Psi_\lambda(x/\lambda)-1}}= \frac{q(1-q^{-\Psi_\lambda(x/\lambda)})}{-1+q}.\]
From here (or using \eref{eq:qdsqrzyu}) it is possible to extract the moments of $\Psi_\lambda$:
\[1,\frac{1}{1-q},
  {\frac {\lambda}{(1-q) 
      \left( \lambda-1 \right) }},
  {\frac {{\lambda}^{2} \left( \lambda+2 \right) }{ \left( \lambda+1 \right)  \left( \lambda-1 \right) ^{2}
 (1-q) }},
{\frac {{\lambda}^{3} \left( {\lambda}^{3}+5\,{\lambda}^{2}+6\,\lambda+6 \right) }{ \left( \lambda+1 \right) 
 \left( \lambda-1 \right) ^{3} \left( {\lambda}^{2}+\lambda+1 \right) 
 (1-q) }}, \cdots\]
We were not able to go further in the description of this distribution.

\begin{OQ} Find a simple description of each of the limit processes $(W(\lambda), \lambda\in I)$ in the binary, geometric, and Poisson cases. 
\end{OQ}
Other cases can be interesting too, but we can expect that these three ones are likely to be the simplest, since they are the simplest model of GW trees.

\section{Proof of Theorem~\ref{thm:Geninf}}
\label{sec:PTTG}

\subsection{Convergence in $D(I,\R^+)$:  tightness under moments assumptions}
\label{sec:tight}

We first give a characterization of convergence in $D(I, \R)$ taken in Billingsley~\cite[Section 13.5]{billin}:
\begin{prop}\label{pro:bill} 
Consider a compact interval $[a,b]\subset \R$, and $(Y_n(\lambda), n\geq 0)_{\lambda\in [a,b]}$ a sequence of processes such that:
  \bir
  \itr
  $(Y_n(\lambda_i), 1\leq i \leq d)\dd (Y(\lambda_i), 1\leq i \leq d)$ for all $d\geq 1$ and $\lambda_1\leq \lambda_2\leq \cdots \leq \lambda_d$ in the set of continuity points of $Y$ on $[a,b]$.
  \itr  $Y(b-\delta)\sur{\to}{(d)} Y(b)$ as $\delta\to 0$, and
  \itr For all $a\leq \lambda_1\leq \lambda_2\leq \lambda_3\leq b$ and $\eta>0$,
  \ben\label{eq:dqgfdiliq}
  \P\big(\min\big\{|\Delta Y_n(\lambda_i)| , i \in\{2,3\} \big\}\geq \eta\big)
  \leq\frac{ (F(\lambda_3)-F(\lambda_1))^{2\alpha}}{\eta^{4\beta}}\een where $\beta\geq 0$, $\alpha > 1/2$, and $F$ is non-decreasing and continuous on $[a,b]$.\\
  \eir
 In this case, $Y_n\dd Y$ in $D([a,b],\R)$.
\end{prop}

\begin{rem} \label{rem:ffqdqs}
By~\cite[Eq.\ (13.14)]{billin}, a sufficient condition for the process
$Y_n(\lambda)=W_n(\lambda)$ (for all $n\geq 0$, $\lambda\in [a,b]$) 
to satisfy \eref{eq:dqgfdiliq} is 
\be
\E\big(|\Delta W_n(\lambda_2)|^2 |\Delta W_n(\lambda_3)|^2\big)
&\leq& {\sf Const}\cdot (\lambda_3-\lambda_1)^{2\alpha}.
\ee
\end{rem}
In fact, we prove that the limit $W$ is in $D(I,\R^+)$ 
as a weak limit of elements of $D(I,\R^+)$ 
for the Skorokhod topology on each compact.

To prove \Cref{thm:Geninf}, we start by a lemma that shows that, in Proposition~\ref{pro:bill}, if the convergence in Assumption (i) holds almost surely, then $Y_n$ converges in probability in $D(I)$ to $Y$.
\begin{lem} \label{lem:grfdq}\label{sec:proba} Assume that a sequence of processes $(T_n,n\geq 0)$ is tight in $D([a,b])$, and that moreover, for any $d\geq 0$ any $a\leq \lambda_1,\cdots,\lambda_d\leq b$,
  the sequence $(T_n(\lambda_1),\cdots,T_n(\lambda_d))$ converges a.s toward some random variables $(T(\lambda_1),\cdots,T(\lambda_d))$ (that are, by construction, consistent). Under these assumptions, $(T_n)$ converges in probability  in $D([a,b])$ (equipped with the Skorokhod topology) to a c\`adl\`ag process $T'$ which coincides with $T$ almost everywhere 
(it is determined by $T$, but can be different on a countable number of points).
\end{lem}
We prove this lemma in \Cref{sec:proof_cvproba}. \par
In \Cref{prp:qsdqs} we stated the a.s. convergence of $(W_n(\lambda_i),1\leq i \leq d)$ for all fixed $\lambda_1<\cdots <\lambda_d \in I$, and then characterised the limiting distribution $(W(\lambda_i),1\leq i \leq d)$. 
By Theorem~\ref{thm:Geninf}, the limiting process in probability of $(W_n)$, which we call $W^*$ for the purpose of this remark, is a c\`adl\`ag process (as the limit of c\`adl\`ag processes in $D(I)$). From this convergence it can only be deduced that the FDD of $W^\star$ are given by those of $(W(\lambda_i),1\leq i \leq d)$ almost everywhere, in fact at the a.s. continuity point of $W^*$ (the complement is a Lebesgue null subset of $I$).

To prove \Cref{thm:Geninf}, it only remains to check that Assumptions $(i)$, $(ii)$
and $(iii)$ of \Cref{pro:bill} are satisfied under the assumptions of~\Cref{thm:Geninf}.
In \Cref{pro:qdsqhsd}, we have already proved that Assumption $(i)$ of \Cref{pro:bill} holds. We prove $(ii)$ in~\Cref{sec:fqfd}, and $(iii)$ in \Cref{sec:iii}. 
In the whole section, we assume that the assumptions of Theorem~\ref{thm:Geninf} hold.

\subsection{The limiting process satisfies Condition $(ii)$ of \Cref{pro:bill}} 
\label{sec:fqfd}
Recall that, by \Cref{pro:qdsqhsd},
\ben\label{eq:tthjy}
\Phi_{\lambda_1,\lambda_2}(x_1,x_2) = \f^{\Delta}_{\lambda_1,\lambda_2}\l[ \Phi_{\lambda_1,\lambda_2}\l(\frac{x_1}{\lambda_1},\frac{x_2}{\lambda_2}\r),\Phi_{\lambda_1,\lambda_2}\l( 0,\frac{x_2}{\lambda_2}\r)\r]
\een
and $\Phi_{\lambda_1, \lambda_2}$ is the unique solution of this equation with mean $(1,1)$ and finite variance.
Differentiating this equation several times in both variables gives equations for the moments of $(W(\lambda_1), W(\lambda_2))$. This can be extended to general $d$ dimensional moments (see also \Cref{sec:ECM}).  
For all integers $k_1, k_2\geq 0$, set
\be
M_{k_1,k_2}:=\E\Big(W(\lambda_1)^{k_1} W(\lambda_2)^{k_2}\Big), 
\ee 
and recall that, by assumption, $M_{0,0}=M_{1,0}=M_{0,1}=1$.
We first aim at proving that
\ben\label{eq:aiamg}
\E((W(\lambda_2)-W(\lambda_1))^2)
=M_{2,0}+M_{0,2}-2M_{1,1}\sous{\longrightarrow}{\lambda_1\to \lambda_2}0,
\een
since it would imply Assumption $(ii)$ of Proposition~\ref{pro:bill}. 
Differentiating~\eref{eq:tthjy}, we get that
\ben\label{eq:rsgf}
\bpar{ccl}M_{2,0}&=& (\Fac{2,0} +\lambda_1 M_{2,0} )/\lambda_1^2\\
M_{0,2}&=&(\lambda_2M_{0, 2}+\Fac{0, 2} +2\Fac{1, 1} +\Fac{2,0})/\lambda_2^2\\
M_{1,1}&=&(\lambda_1M_{1, 1}+\Fac{1, 1} +\Fac{2, 0} )/(\lambda_1\lambda_2),
\epar\een
where we have for all integers $\beta_1, \ldots, \beta_d\geq 0$,
\ben\label{eq:dqdhhuds}
\Fac{\beta\cro{1,d}}(\lambda\cro{1,d})
:=\frac{\partial^{|\beta|}}{\partial y_1^{\beta_1}\cdots \partial y_d^{\beta_d}}\f^{\Delta}(y_1,\cdots,y_d)\Big|_{[1,\cdots,1]}
\een
From~\eqref{eq:rsgf}, we get
\ben\label{eq:dqdsdq2}
M_{2,0} ={\frac {\Fac{2,0}}{\lambda_1 \left( \lambda_1 -1\right) }},~~ 
M_{0,2}= \frac {\Fac{0,2}+2\,\Fac{1,1}+\Fac{2,0}}{\lambda_2 \left( \lambda_2-1 \right) },~~ 
M_{1,1}=\frac {\Fac{1,1}+\Fac{2,0}}{\lambda_1 \left( \lambda_2-1 \right)}.
\een
Thus, by~\eqref{eq:aiamg}, we have, as $\lambda_2\to\lambda_1$,
\[\E\l((W(\lambda_2)-W(\lambda_1))^2\r)
\sim \frac{\Fac{0,2}}{\lambda_1(\lambda_1-1)}.
\]
Thus, for the right-hand side of~\eref{eq:aiamg} to tend to zero as $\lambda_2\to\lambda_1$, it is necessary and sufficient that ${\sf Fac}^{\Delta}_{{0,2}} \to 0$, which is implied by~{\HYPM}, Equation~\eqref{eq:dsdity}.

\subsection{The sequence $(W_n)$ satisfies Condition $(iii)$ of \Cref{pro:bill}}
\label{sec:iii}
{The last remaining step to prove the tightness is the following result.}
\begin{thm}\label{th:tightness} 
  If {\HYPR} and {\HYPM} hold, 
 then  for all $[a,b]\subseteq I$,
there exists a constant $C_W>0$ such that, for all $\lambda_1\leq \lambda_2 \leq \lambda_3\in [a,b]$,
\be\label{eq:fqdqksd}\sup_{n\geq 0} \mathbb E\l[\l(\Delta W_n(\lambda_2)\r)^2\l(\Delta W_n(\lambda_3)\r)^2\r]\leq C_W(\lambda_3-\lambda_1)^{2\kappa}.
\ee
\end{thm} 
{By \Cref{rem:ffqdqs}, this theorem implies that the sequence $(W_n)$ satisfies Condition $(iii)$ of \Cref{pro:bill}. The proof of \Cref{th:tightness} is quite long: it will last until the end of Section \ref{sec:PTTG}.}

\paragraph{The important ``$S$'' notation:}
Fix an interval $[a,b]\in I$ and some $ \lambda_1 \leq \lambda_2\leq \lambda_3 \in[a,b]$. 
For $r\in\{1,2,3\}$ and $j\in\{1,2,3\}$, set
\begin{equation}\label{eq:defS}
S_r(W_n(\lambda_j)):=\sum_{i=1+X(\lambda_{r-1})}^{X(\lambda_r)} W_n^{(i)}(\lambda_j),
\end{equation}
where $(W_n^{(i)}(\lambda_j)\colon n\geq 0)_{i\geq 1}$ is a sequence of i.i.d.\ copies of $(W_n(\lambda_j)\colon n\geq 0)$.
Also recall that we set $\lambda_0 = 0$.
For example, we can write
\ben \label{eq:rel}
W_{n+1}(\lambda_j) 
= \frac{1}{\lambda_j}\sum_{i=1}^{X(\lambda_j)} W_n^{(i)}(\lambda_j)
= \sum_{\ell =1}^j \frac{\S{\ell}{W_n(\lambda_j)}}{\lambda_j}.
\een
We use the fact that, for all $i\neq j$, $W_n^{(i)}$ is independent of $W_n^{(j)}$, 
to get the following lemma:
\begin{lem} \label{lem:propS}
For any $\lambda_j, \lambda_\ell \in \{\lambda_1, \lambda_2, \lambda_3\}$,
$\S{r}{W_n(\lambda_j)}$ and $\S{s}{W_n(\lambda_\ell)}$ are independent iff $r\neq s$.
Moreover, $S_r$ is linear in the following sense: 
for all constants $c_1 ,c_2, c_3$,
\[S_r\bigg(\sum_{j=1}^3 c_j W_n(\lambda_j)\bigg)
=\sum_{j=1}^3c_j\S{r}{W_n(\lambda_j)}.\]
\end{lem}

Using this notation in~\eref{eq:rel}, we get
%\begin{linenomath}
  \begin{align}
\label{eq:dqgfsds}
\Delta W_{n+1}(\lambda_2) 
&=\frac{\S{1}{W_n(\lambda_2)}}{\lambda_2}+\frac{\S{2}{W_n(\lambda_2)}}{\lambda_2}- \frac{\S{1}{W_n(\lambda_1)}}{\lambda_1},\\
\Delta W_{n+1}(\lambda_3) 
&=\frac{\S{1}{W_n(\lambda_3)}}{\lambda_3}+\frac{\S{2}{W(\lambda_3)}}{\lambda_3}+\frac{\S{3}{W(\lambda_3)}}{\lambda_3}- \frac{\S{1}{W_n(\lambda_2)}}{\lambda_2}- \frac{\S{2}{W_n(\lambda_2)}}{\lambda_2}.
\end{align}%\end{linenomath}
Using~\eref{eq:rel} and the linearity of $S_r$, we can write
%\begin{linenomath}
  \begin{align} \label{eq:fqqd1}
\Delta W_{n+1}(\lambda_2) &=  \frac{1}{\lambda_2}\l[ T_1+T_2+T_3 \r],\\
\label{eq:fqqd2}\Delta W_{n+1}(\lambda_3) &=  \frac{1}{\lambda_3}\l[ T'_1+T'_2+T'_3+T'_4+T'_5\r],
\end{align}%\end{linenomath}
where we have set
\begin{equation}\label{eq:theTi}\left\{
\begin{array}{ccl}
             T_1=\S{1}{\Delta W_n(\lambda_2)}, &T_2 =\S{1}{- \frac{\Delta \lambda_2}{\lambda_1}W_n(\lambda_1)}, & T_3=\S{2}{W_n(\lambda_2)},\\[3pt]
T'_1= \S{1}{\Delta W_n(\lambda_3)}, &T'_2=\S{1}{- \frac{\Delta \lambda_3}{\lambda_2}W_n(\lambda_2)}, & T'_3=\S{2}{\Delta W_n(\lambda_3)},\\[3pt]
& T'_4= \S{2}{- \frac{\Delta \lambda_3}{\lambda_2}W_n(\lambda_2)}, & T'_5=\S{3}{W_n(\lambda_3)}.
  \end{array}\right.\end{equation}
The reason why we decompose $\Delta W_{n+1}(\lambda_2)$ and $\Delta W_{n+1}(\lambda_3)$ this way is because it maximises the number of $\Delta$'s; this is important because, intuitively, $\Delta$'s give terms that are small when $|\lambda_3-\lambda_1|$ goes to zero.

\paragraph{Proof strategy:} To prove Assumption~$(iii)$ of Proposition~\ref{pro:bill},
we start by writing 
\ben\label{eq:decT}
 \E\l(\Delta W_{n+1}(\lambda_2)^2\Delta W_{n+1}(\lambda_3)^2\r)=\frac{1}{(\lambda_2\lambda_3)^2}\sum_{1\leq i_1,i_2\leq 3 \atop{1\leq j_1,j_2\leq 5}} \E\l(T_{i_1}T_{i_2} T'_{j_1}T'_{j_2}\r).\een
Note that there are $\bs {9\times 25=225}$ terms in this sum, which we call ``$T$-moments'' from now on.
Our strategy is to analyse the contribution of each of these $T$-moments. 
To do so, we will first expand each of the 225 $T$-moments using Lemma \ref{lem:Mom} below. 
This will give
\ben\label{eq:final_rec}
\E\l(\Delta W_{n+1}(\lambda_2)^2\Delta W_{n+1}(\lambda_3)^2\r)
=\frac{\lambda_1}{(\lambda_2\lambda_3)^2}\E\l(\Delta W_{n}(\lambda_2)^2\Delta W_{n}(\lambda_3)^2\r)+\sum_{m} {\sf Term}_m
\een
where the index~$m$ ranges over several hundreds of values 
(there are more terms in this sum than the 225 initial terms of~\eqref{eq:decT}).
From now one, we call the ${\sf Term}_m$ the ``multinomials''.
Importantly, the sum in~\eqref{eq:final_rec} is finite, and each of the multinomials satisfies
\ben\label{eq:bound_terms}
{\sf Terms}_m\leq c_m (\lambda_3-\lambda_1)^{2\kappa}
\een 
for a finite constant $c_m$ (which depends on $m$).
To show~\eqref{eq:bound_terms}, we do not treat the several hundred multinomials one by one. 
Instead, we partition them in several families, and show that all multinomials in each of these families satisfy~\eqref{eq:bound_terms}.

Using~\eref{eq:bound_terms} in~\eref{eq:final_rec}, we get
\ben\label{eq:lastlast_rec}
\E\l(\Delta W_{n+1}(\lambda_2)^2\Delta W_{n+1}(\lambda_3)^2\r)\leq \frac{\lambda_1}{(\lambda_2\lambda_3)^2}\E\l(\Delta W_{n}(\lambda_2)^2\Delta W_{n}(\lambda_3)^2\r)+\sum_{m} c_m (\lambda_3-\lambda_1)^{2\kappa},
\een
which implies that for $[a,b]$ fixed, 
there exist some universal constants $A\in(0,1)$ and $C>0$ such that, for all $\lambda_1<\lambda_2<\lambda_3\in [a,b]$,
\[\E(\Delta W_{n+1}(\lambda_2)^2\Delta W_{n+1}(\lambda_3)^2) \leq A \E(\Delta W_{n+1}(\lambda_2)^2\Delta W_{n+1}(\lambda_3)^2)+ C(\lambda_3-\lambda_1)^{2\kappa}.\] 
Iterating this formula (see Lemma \ref{lem:rec} below, which can be applied since $\Delta W_0(\lambda_2)\Delta W_0(\lambda_3)=0$ a.s.) allows us to conclude that
\[\sup_{n\geq 0} \E\l(\Delta W_{n+1}(\lambda_2)^2\Delta W_{n+1}(\lambda_3)^2\r) \leq C_W(\lambda_3-\lambda_1)^{2\kappa}\]
for a finite constant $C_{W}>0$, as claimed. 
\begin{lem}\label{lem:rec}
Let $(U_n)_{n\geq 0}$ be a sequence of non-negative real numbers such that, for all $n\geq 0$,
\begin{equation}\label{eq:suite}
U_{n+1}\leq AU_n+B,
\end{equation}
for some constants $A\in [0,1)$ and $B\geq 0$. In this case, for all $n\geq 0$, 
$U_n\leq A^n U_0 + \frac B{1-A}$.
\end{lem}
\begin{proof}
Iterating~\eqref{eq:suite}, we get 
$U_n \leq AU_{n-1}+B \leq A(AU_{n-2}+B) + B \leq A^n U_0 + B\sum_{i=0}^{n-1} A^i$.
\end{proof}
Thus, to conclude the proof, it only remains to prove~\eref{eq:final_rec} and~\eref{eq:bound_terms}.
To do this, we first describe the multinomials that appear in~\eqref{eq:final_rec}; this is done in Section~\ref{sec:AE} by expanding each of the $T$-moments.
We then show why each of these multinomials can be bounded by $c_m (\lambda_3-\lambda_1)^{2\kappa}$ (and thus why~\eqref{eq:bound_terms} holds); this is done in Section~\ref{sec:cla} by classifying the 225 of $T$-terms into four different classes.

\subsection{Algebraic expansion of $\E(\Delta W_{n+1}(\lambda_2)^2 \Delta W_{n+1}(\lambda_3)^2)$}
\label{sec:AE}

As already said, we write $\Fac{d_1,d_2,d_3}$ instead of $\Fac{d_1,d_2,d_3}(\lambda_1,\lambda_2,\lambda_3)$.

For a finite set $B$, we let ${\sf Part}(B,k)$ be the set of partitions of $B$ into $k$ non empty parts (the parts must be disjoint, and their union must be $B$). 
For example,
    \be
    {\sf Part}(\{1,2,3,4\},3)&=&\Big\{ [\{1\},\{2\},\{3,4\}],[\{1\},\{2,3\},\{4\}],[\{1\},\{2,4\},\{3\}],\\
    && [\{1,2\},\{3\},\{4\}],[\{1,3\},\{2\},\{4\}],[\{1,4\},\{2\},\{3\}]\Big\}.
    \ee
A partition is formally a set of sets. We consider the canonical representation of a partition as a sequence of sets, where the sequence is obtained by sorting the sets according to their smallest element, as done in the example above.

In the next lemma, we show how one can express each $T$-moment $\E(T_{i_1}T_{i_2}T'_{j_1}T'_{j_2})$ as a linear combination of moments of $(W_n(\lambda_1), W_n(\lambda_2), W_n(\lambda_3))$.
\begin{lem} \label{lem:Mom}
Let $n\geq 1$ and $(V_{n,j})_{j\geq 1}$ such that, for $j\geq 1$, $V_{n,j}\in\big\{W_n(\lambda_\ell), \Delta W_n(\lambda_\ell) \colon 1\leq\ell\leq 3\big\}$.

The next formula hold if the moments involved are well defined.\par
Assume that $B_1$, $B_2$ and $B_3$ are disjoint set of indices. 
We have 
\begin{equation}\label{eq:exp_part}
\E\l[\prod_{j\in B_1}S_1(V_{n,j}) \prod_{k\in B_2}S_2(V_{n,k})\prod_{\ell\in B_3}S_3(V_{n,\ell})\r]
= \sum_{d_1\geq 0\atop {d_2\geq 0 \atop d_3\geq 0}}{\sf Fac}^{\Delta}_{d_1,d_2,d_3}\prod_{\ell=1}^{3}\sum_{[A_1,\ldots,A_{d_\ell}]\atop \in{\sf Part}(B_\ell,d_\ell)} \prod_{k=1}^{d_\ell}\E\l(\prod_{s\in A_k}V_{n,s}\r)
\end{equation}
\end{lem}

We show on three particular examples how this formula can be applied to each of the $T$-moments:

\medskip
\textbf{Example 1.} We use Lemma~\ref{lem:Mom} to expand 
\[Q_1:=\E\l(\S{1}{W_n(\lambda_1)}^2\S1{\Delta W_n(\lambda_2)}\,\S3{ W_n(\lambda_3)}\r).\]
We set $V_{n,1} = V_{n,2} = W_n(\lambda_1)$, $V_{n,3}=\Delta W_n(\lambda_2)$, and $V_{n,4}=W_n(\lambda_3)$. We also set $B_1=\{1,2,3\}$, $B_2=\varnothing$, and $B_3=\{4\}$. 
With this notation, $Q_1$ is indeed equal to the left-hand side of~\eqref{eq:exp_part}. 

We now look at the right-hand side of~\eqref{eq:exp_part}. First note that $B_1$ can be partitioned into 1, 2, or 3 parts, i.e.\ $d_1$ ranges from 1 to 3 in the right-hand side of~\eqref{eq:exp_part}. We have ${\sf Part}(B_1, 1)=\{ [\{1,2,3\}]\}$, ${\sf Part}(B_1,2)=\{ [\{1,2\},\{3\}],[\{1\},\{2,3\}],[\{1,3\},\{2\}]\}$, and ${\sf Part}(B_1,3)=\{ [\{1\},\{2\},\{3\}]\}$.
Similarly, $d_2 = 0$ and (by convention) ${\sf Part}(B_2,0)= \{[\varnothing]\}$.
Finally, $d_3 = 1$, and ${\sf Part}(B_3,1)=\{[\{4\}]\}$.
Applying~\eqref{eq:exp_part}, we thus get
\be
Q_1&=&{\sf Fac}^{\Delta}_{1,0,1} \l[\E( V_{n,1}V_{n,2}V_{n,3})\r]\l[\E(V_{n,4})\r]\\
&+&{\sf Fac}^{\Delta}_{2,0,1} \l[\E( V_{n,1}V_{n,2})\E(V_{n,3})+\E( V_{n,1})\E(V_{n,2}V_{n,3})+\E( V_{n,1}V_{n,3})\E(V_{n,2})\r]\l[\E(V_{n,4})\r]\\
&+&{\sf Fac}^{\Delta}_{3,0,1}\l[\E( V_{n,1})\E(V_{n,2})\E(V_{n,3})\r]\l[\E(V_{n,4})\r]
\ee
In the second term of the sum, since $d_1 = 2$ in $\Fac{d_1,d_2,d_3}$ we separate the product $\prod_{j\in B_1} V_{n,j}$ into two independent non-empty products. There are three possible ways to do that, and they give the following sum of three terms: $\E(V_{n,1}V_{n,2})\E(V_{n,3})+\E(V_{n,1})\E(V_{n,2}V_{n,3})+\E( V_{n,1}V_{n,3})\E(V_{n,2})$.

\medskip
{\bf Example 2.} We now show how to apply Lemma~\ref{lem:Mom} to expand 
\[Q_2 := \E\l(\S{2}{W_n(\lambda_3)}\S{3}{\Delta W_n(\lambda_1)}\r).\] 
To do so, we set $V_{n,1}=W_n(\lambda_3)$, $V_{n,2}=\Delta W_n(\lambda_1)$, $B_1=\varnothing, B_2= \{1\}$, and $B_3=\{2\}$. 
With these definitions, $Q_2$ equals the left-hand side of~\eqref{eq:exp_part}.
We now look at the right-hand side of~\eqref{eq:exp_part}: 
Since $B_1$ is empty, and since both $B_2$ and $B_3$ have one element, the only possibility is $d_1 = 0$ and $d_2 = d_3 = 1$. We thus get 
\[\E\l(\S{2}{W_n(\lambda_3)}\S{3}{\Delta W_n(\lambda_1)}\r)
=\Fac{0,1,1} \E(W_n(\lambda_3))\E(\Delta W_n(\lambda_1))=0,\]
since $\E(\Delta W_n(\lambda_1))=0$.
This is not surprising; indeed, we have
\[Q_2 = \E\l(\sum_{i=X(\lambda_{1})+1}^{X(\lambda_2)} W_n^{(i)}(\lambda_3) \sum_{j=X(\lambda_{2})+1}^{X(\lambda_3)} \Delta W_n^{(j)}(\lambda_1)\r).\]
Because the sequence $(W^{(i)}_n(\lambda)\colon n\geq 0, \lambda>1)_{i\geq 1}$ 
is a sequence of i.i.d.\ copies of $(W_n(\lambda)\colon n\geq 0, \lambda>1)$, we indeed get
\[Q_2 = \E\big[(\Delta X(\lambda_2))(\Delta X(\lambda_1))\big] \E(W_n(\lambda_3))
\E(\Delta W_n(\lambda_1))=0.\]

\medskip
{\bf Example 3.} We show how to use Lemma~\ref{lem:Mom} to expand 
\[Q_3:=\E\l(\S{3}{W_n(\lambda_3)}\S{3}{\Delta W_n(\lambda_1)}\r).\] 
We set $V_{n,1} = W_n(\lambda_3)$, $V_{n,2} = \Delta W_n(\lambda_1)$, $B_1=B_2=\varnothing$ and $B_{3}=\{1,2\}$, so that $Q_3$ is indeed of the form of the left-hand side of~\eqref{eq:exp_part}. 
We have ${\sf Part}(B_3,1)=\{[\{1,2\}]\}$ and ${\sf Part}(B_3,2)=\{[\{1\},\{2\}]\}$, so that $d_1=0,d_2=0, 1\leq d_3\leq 2$. We thus get
\be
\E\l(\S{3}{W_n(\lambda_3)}\S{3}{\Delta W_n(\lambda_1)}\r)&=& \Fac{0,0,1}\E\l(W_n(\lambda_3)\Delta W_n(\lambda_1)\r) + \Fac{0,0,2}\E\l(W_n(\lambda_3)\r)\E\l(\Delta W_n(\lambda_1)\r)\\
&=&\Fac{0,0,1}\E\l(W_n(\lambda_3)\Delta W_n(\lambda_1)\r)
\ee
Again, this can be checked directly by computing 
\[\E\l(\sum_{i=X(\lambda_{2})+1}^{X(\lambda_3)} W_n^{(i)}(\lambda_3) \sum_{i=X(\lambda_{2})+1}^{X(\lambda_3)} \Delta W_n^{(i)}(\lambda_1)\r)\]
and by regrouping the terms involving the same ${(i)}$ and the others.
The advantage of Lemma~\ref{lem:Mom} is to give a general formula that applies to all of the $\E(T_{i_1}T_{i_2}T'_{j_1}T'_{j_2})$.

\begin{proof}[Proof of Lemma \ref{lem:Mom}]
Because of Lemma~\ref{lem:propS}, conditionally on  $\{\Delta X_{\lambda_j}=x_j, 1\leq j \leq 3\}$, the three products inside the expectation on the left-hand side are independent.
We thus need to calculate 
$\E\l(\prod_{j\in B_\ell}S_\ell(V_{n,j}) ~|~\Delta X(\lambda_\ell)\r)$.
Now recall that, by definition of $S_\ell$ (see~\eqref{eq:defS}), we have
\[\prod_{j\in B_\ell}S_\ell(V_{n,j}) 
= \prod_{j\in B_\ell}\sum_{i=X(\lambda_{\ell-1})+1}^{X(\lambda_\ell)} V^{(i)}_{n,j}
= \sum_{i_1=X(\lambda_{\ell-1})+1}^{X(\lambda_\ell)}\cdots 
\sum_{i_m=X(\lambda_{\ell-1})+1}^{X(\lambda_\ell)} V^{(i_1)}_{n,j_1}\cdots V^{(i_m)}_{n,j_m},\]
where $m=m(\ell)$ is the cardinal of $B_\ell = \{j_1, \ldots, j_m\}$.
Shifting the indices from the range $[X(\lambda_{\ell-1})+1, X(\lambda_{\ell})]$ to $[1, \Delta X(\lambda_\ell)]$ does not affect the distribution of the right-hand side, implying that
\[\E\Bigg(\prod_{j\in B_\ell}S_\ell(V_{n,j})~\mid~~\Delta X(\lambda_\ell)=x\Bigg)
= \sum_{i_1=1}^{x}\cdots 
\sum_{i_m=1}^{x} \E\big(V^{(i_1)}_{n,j_1}\cdots V^{(i_m)}_{n,j_m}\big).
\]
We now re-write this sum by grouping 
the indices $i_1, \ldots, i_m$ that are equal and using independence when the indices differ. 
For all $k\in \{1, \ldots, x\}$, group all $j_\ell$'s such that $i_\ell = k$ into one (possibly empty) part.
This forms a partition of $B_\ell$. We decompose the sum above depending on the number of non-empty parts in this partition of $m$, which we call~$d$: this gives
\[\E\l(\prod_{j\in B_\ell}\S{\ell}{V_{n,j}} ~|~\Delta X(\lambda_\ell)=x\r)
= \sum_{d\geq 1} x(x-1)\cdots (x-d+1)\sum_{[A_1,\cdots,A_{d}]\in{\sf Part}(B_\ell, d)} 
\prod_{k=1}^{d}\E\l(\prod_{\ell\in A_k}V_{n,\ell}\r).\]
The factor $x(x-1)\cdots (x-d+1)$ is the number of different choices for the common index~$i$ for the first, second, etc parts of the partition: there are $x$ choices for the first part, $x-1$ choices for the second part, and so on.
\end{proof}

\subsubsection{Expansion of $\E(\Delta W_{n+1}(\lambda_2)^2\Delta W_{n+1}(\lambda_2)^3)$: classification of the contributions}
\label{sec:graphical}
The aim of this section is to show how to apply Lemma~\ref{lem:Mom} to each of the $T$-moments that appear in~\eqref{eq:decT}.
This allows us to eventually write 
$\E(\Delta W_{n+1}(\lambda_2)^2\Delta W_{n+1}(\lambda_3)^2)$ 
as in~\eqref{eq:final_rec}, where each of the multinomials is a coefficient $\Fac{d_1, d_2, d_3}$ 
times a product of ``monomials''~$\E(\prod_{r\in A_k}V_{n,r})$.
Since we see $\E(\prod_{r\in A_k}V_{n,r})$ as a monomial, we will call $|A_k|$ its ``degree''.
If the degree is~1, then, because all the involved $V_{n,j}$ belongs to $\cup_{i=1}^3\{W_n(\lambda_i), \Delta W_n(\lambda_i)\}$ (see \eref{eq:fqqd1} and \eref{eq:fqqd2}), 
$\E(V_{n,j})$ is either 1 or 0 (because $\E(W_n(\lambda_i))=1$ and $\E(\Delta W_n(\lambda_i))=0$).
Monomials of degrees 2 correspond to correlations, 
and degree 3 and 4 are the most difficult monomials to handle in our analysis.

\paragraph{Graphical representation of the complete computation}
First, represent in an array, as in Fig.\ \ref{fig:rep0}, the $T_i$ and the $T'_i$ as defined in \eref{eq:theTi}.
\begin{figure}[!htbp]
\centerline{\includegraphics{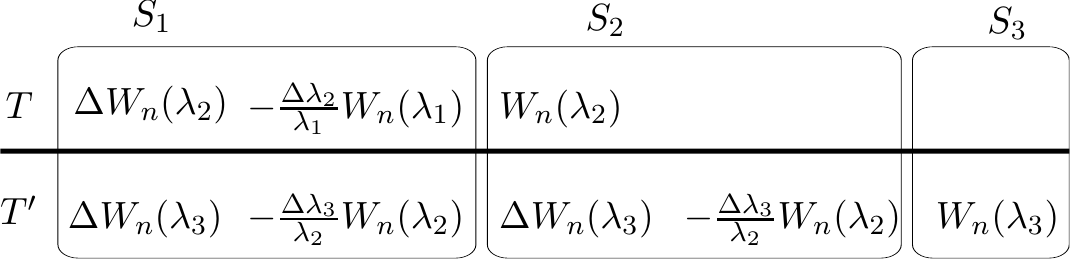}}
\caption{\label{fig:rep0} On the top line $T_1$, $T_2$, $T_3$ in this order (for example $T_1=\S{1}{\Delta W_n(\lambda_2)}$); on the second line, the $T'_1, T'_2,T'_3,T'_4,T'_5$ in this order, for example, $T_3'=\S{2}{\Delta W_n(\lambda_3)}$.}
\end{figure}
With this graphical representation, each element $(T_{i_1},T_{i_2},T'_{j_1},T'_{j_2})$ of the sum in~\eqref{eq:decT} can be obtained by multiplying an ordered pair of elements (with repetition allowed) above the line and an ordered pair of elements below the line. Two examples are given below:
\begin{figure}[!htbp]
\centerline{\includegraphics{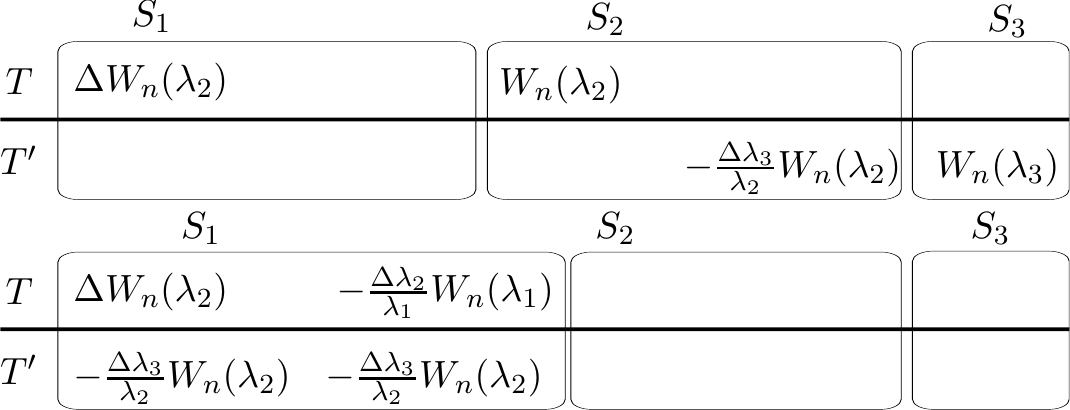}}
\caption{Two examples of choices for $(i_1, i_2, j_1, j_2)$: in the first case on the top, we have selected $(T_1,T_3,T'_4,T'_5)$ and, in the second case on the bottom we have selected $(T_1,T_2, T'_2, T'_2)$ (repetitions are allowed: here, we have chosen $j_1 = j_2 = 2$).}
\label{fig:rep4}
\end{figure}
Each of the $T$-moments can be represented using this graphical tool, and this becomes useful when applying Lemma \ref{lem:Mom}.
Indeed, in this graphical representation, we can see that the four terms $T_{i_1}, T_{i_2}, T'_{j_1}, T'_{j_2}$ are partitioned into three (possibly empty) groups: the $S_1$-group, the $S_2$-group and the $S_3$-group, represented graphically by the three rounded rectangles.
In the right-hand side of~\eqref{eq:exp_part}, we consider all possible ways to refine this partition, 
meaning that each part of the chosen partition must be included in one of the three rounded rectangles, as for example in Fig.\ \ref{fig:rep5}.
\begin{figure}[!htbp]
\centerline{\includegraphics{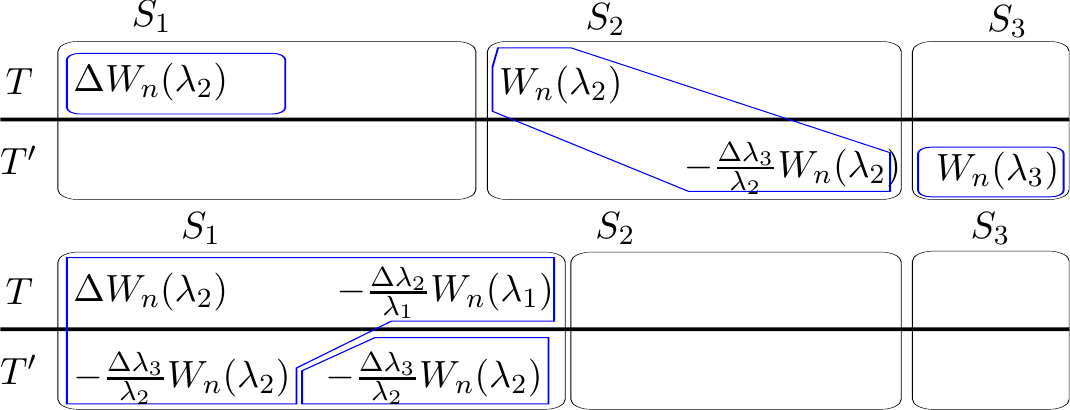}} \caption{\label{fig:rep5} Two refined partitions of, respectively, the top and bottom examples in Fig.\ \ref{fig:rep4}.}
\end{figure}\par
If, in a refined partition, $S_1$ is split in $d_1$ parts, $S_2$ in $d_2$ parts, and $S_3$ in $d_3$ parts, then the contribution of this partition $\Pi$ to the right-hand side of~\eqref{eq:exp_part} is the multinomial
 \[\Fac{d_1,d_2,d_3}\prod_{P \in \Pi} \E\Bigg(\prod_{e \in P} e\Bigg),\]
 where we sum on all the parts of the refined partition $\Pi$, and then multiply on all elements of this part~$P$. Note that $d_1, d_2\leq 4$, and $d_3\leq 2$, since there are maximum 4 terms in the rounded rectangle associated to $S_1$, resp.\ $S_2$, and maximum 2 terms in the rounded rectangle associated to~$S_3$. Furthermore, $d_1+d_2+d_3=4$ since there are four terms in total: $T_{i_1}, T_{i_2}, T'_{j_1}$, and $T'_{j_2}$.

For example, the contribution of the refined partition on the top of Fig.\ \ref{fig:rep5} is the multinomial
 \[\Fac{1,1,1} \E\l[ \Delta W_n(\lambda_2)\r] \E\l[\Delta W_n(\lambda_2)\l(- \frac{\Delta \lambda_3}{\lambda_2}W_n(\lambda_2)\r)\r]\E\l[W_n(\lambda_3)\r],\]
and the contribution of the refined partition on the bottom of Fig.\ \ref{fig:rep5} is the multinomial
\[\Fac{2,0,0} \E\l[ \Delta W_n(\lambda_2)\l(- \frac{\Delta \lambda_2}{\lambda_1}W_n(\lambda_1)\r)  \l(- \frac{\Delta \lambda_3}{\lambda_2}W_n(\lambda_2)\r)\r]\E\l[- \frac{\Delta \lambda_3}{\lambda_2}W_n(\lambda_2)\r].\]

From this graphical representation, one can see that the only way to get a term that contains 
$\E\l(\Delta W_{n}(\lambda_2)^2\Delta W_{n}(\lambda_3)^2\r)$ is to have 4 terms of type $S_1$ in the same part of the refined partition. 
This only occurs in the development of the $T$-moment $\E(T_1,T_1,T'_1,T'_1)$, and only for $d_1 = 1$, $d_2 = d_3 = 0$. Thus, from~\eqref{eq:decT}, we get
\[\E\l(\Delta W_{n+1}(\lambda_2)^2\Delta W_{n+1}(\lambda_3)^2\r)
= \frac{\Fac{1,0,0}}{(\lambda_2\lambda_3)^2}
\E\l(\Delta W_{n}(\lambda_2)^2\Delta W_{n}(\lambda_3)^2\r)
+ \sum_{m}{\sf Term}_m,\]
and the multinomial $\E\l(\Delta W_{n}(\lambda_2)^2\Delta W_{n}(\lambda_3)^2\r)$ does not appear in any of the multinomials in the sum.
Because $\Fac{1,0,0}=\lambda_1$, this gives~\eref{eq:final_rec}.

\subsubsection{Conclusion}\label{sec:cla}
To conclude the proof, it only remains to bound all of the multinomials by $c_m(\lambda_3-\lambda_1)^{2\kappa}$, as announced in~\eqref{eq:bound_terms}.
To do so, we need the following lemma:
\begin{lem}\label{lem:bounds_lower_moms}
Fix a compact subinterval $[a,b]$ of $I$. If {\HYPR} and {\HYPM} hold, then
\begin{enumerate}[(i)]
\item \label{it:a01} For any $a\leq \lambda_1\leq \lambda_2\leq \lambda_3\leq b$,
\[ \Delta \lambda_2 \leq (\lambda_3-\lambda_1), ~~\Delta \lambda_3 \leq (\lambda_3-\lambda_1).\]
    \item  \label{it:a1} For any $k_1,k_2,k_3\in\{0,1,2,\cdots\}$ such that $k_1+k_2+k_3\leq 4$,
      \be
      \ov{M}_{k_1,k_2,k_3}:= \sup_{n} \sup_{a\leq \lambda_1\leq \lambda_2\leq \lambda_3\leq b}\E\l[ W_n(\lambda_1)^{k_1}W_n(\lambda_2)^{k_2}W_n(\lambda_3)^{k_3}\r]<+\infty\ee
     
      \item  \label{it:a4}  If $1\leq k_1+k_2 \leq 3$ and $k_1+k_2+j_1+j_2+j_3\leq 4$ then, there exists a constant $C\geq 0$, such that for all $a\leq \lambda_1\leq \lambda_2\leq \lambda_3\leq b$,
        \[\sup_{n\geq 0} \l|\mathbb E\l[(\Delta W_n(\lambda_2))^{k_1}\,(\Delta W_n(\lambda_3))^{k_2}\,W_n(\lambda_1)^{j_1}\,W_n(\lambda_2)^{j_2}\,W_n^{j_3}(\lambda_3)\r]\r|\leq C(\lambda_3-\lambda_1)^\kappa.\]        \end{enumerate}
    \end{lem}
    \begin{proof}
    $(i)$ is straightforward. We prove $(ii)$ in Section~\ref{sub:proof_key_lemma1}, and $(iii)$ in Section~\ref{sub:proof_key_lemma2}.
    \end{proof}

To prove the bounds of~\eref{eq:bound_terms}, we use the graphical representation of Section~\ref{sec:graphical}. 
Each of the multinomials comes from the expansion 
of a $T$-moment.
We divide the $T$-moments into three groups: the $T$-moments that involve at least one element from the rounded rectangle associated to $S_3$ (i.e.\ $j_1=5$ or $j_2 =5$), the $T$-moments that involve no element from $S_3$ but at least one element from $S_2$, and finally, the $T$-moments that involve no element from $S_2$ or $S_3$.

\paragraph{$T$-moments that involve elements from $S_3$.} 
Note that the only term from $S_3$ is $T'_5$, 
and it can appear once or twice in a $T$-moment $\E(T_{i_1}T_{i_2}T'_{j_1}T'_{j_2})$. 
Apart from $T'_5$, only $T_3$, which is an element from $S_2$, does not contain the symbol $\Delta$. 

Recall that, intuitively, the terms containing a $\Delta$ go to zero when $\lambda_3$ tends to $\lambda_1$. Therefore, intuitively, the worst possible case is $\E(T_3^2 (T'_5)^2)$. We expand this using Lemma~\ref{lem:Mom}: 
we get a sum of multinomials, which are the product of a prefactors $\Fac{0,y,z}$ with $1\leq y, z \leq 2$ times a product of monomials of degree at most~4 in $W_n(\lambda_2)$ and $W_n(\lambda_3)$. 
By~\eref{eq:dsdity2} and Lemma~\ref{lem:bounds_lower_moms}{\it\eref{it:a1}} these multinomials are indeed all bounded by $C(\lambda_3-\lambda_1)^{2\kappa}$ as claimed in~\eqref{eq:bound_terms}.

Now assume that $j_1 = 5$ or $j_2=5$, and $T_{i_1}T_{i_2}T'_{j_1}T'_{j_2}\neq T_3^2(T'_5)^2$.
In view of Fig.\ \ref{fig:rep0} this means that one of $T_{i_1}, T_{i_2}, T'_{j_1}$, and $T'_{j_2}$ is either:
\begin{itemize}
    \item  a term of the form $\Delta W_n(\lambda_\ell)$; this term can be a $S_1$ or a $S_2$-term: in any case by Lemma \ref{lem:bounds_lower_moms}{\it\eref{it:a4}}, any monomial containing such a term is bounded by $C(\lambda_3-\lambda_1)^{\kappa}$.
    \item a term that contains a~$\Delta \lambda_j$; using Lemma~\ref{lem:bounds_lower_moms}{\it\eref{it:a1}}, any monomial containing at least one of these terms is bounded in absolute value by $C(\lambda_3-\lambda_1)$ for a universal constant $C$.
    \end{itemize}
In both cases, the prefactor $\Fac{x,y,z}$ with $z\in\{1,2\}$ brings the extra term $C(\lambda_3-\lambda_1)^{\kappa}$ needed by \eref{eq:dsdity} or \eref{eq:dsdity2}.

\paragraph{Terms that involve no $S_3$ terms but at least one $S_2$-term.}
The only term in the $S_1$ or $S_2$ group that comes without any $\Delta$ is $T_3$;
thus, the ``worst case'' for a $T$-moment in this group is to have $i_1 = i_2 = 3$.
All multinomials obtained when expanding such a $T$-moment come with a prefactor $\Fac{d_1, d_2, 0}$ with $0\leq d_1, d_2\leq 4$ and $d_1+d_2\leq 4$. Also note that since the $T$-factor contains at least one term from $S_2$, it contains at most three terms from $S_1$. This implies $d_2\geq 1$ and $d_1\leq 3$ (also, $d_3 = 0$ because this $T$-moment contains no $S_3$-term). We distinguish cases according to the value of $d_1$. 
First note that, by {\HYPM}, since $d_2\neq 0$, we have $\Fac{d_1, d_2, 0}\leq C(\lambda_3-\lambda_1)^{\kappa}$.
\begin{itemize}
\item If $d_1 = 3$, then the corresponding multinomial (without its prefactor) is a product of three expectations, each of one~$S_1$-term, times the expectation of an $S_2$-term. 
All of the first three expectations contain a $\Delta$ and are thus bounded by $C(\lambda_3-\lambda_1)^{\kappa}$, by Lemma~\ref{lem:bounds_lower_moms}$(i)$ and $(iii)$. 
The fourth expectation is bounded by a constant by Lemma~\ref{lem:bounds_lower_moms}$(ii)$ 
(and the triangular inequality if the term from $S_2$ contains a $\Delta$). 
In total, with its prefactor, the multinomial is thus bounded by $C(\lambda_3-\lambda_1)^{2\kappa}$, as claimed.
\item If $d_1 = 2$, then the corresponding multinomial is a product of three (if $d_2 = 1$) or four (if $d_2=2$) expectations: two of these are expectations of $S_1$-terms, the other one or two are expectations of $S_2$-terms. 
We bound the expectations of $S_2$-terms by constants using Lemma~\ref{lem:bounds_lower_moms}$(ii)$ (and the triangular inequality if they contain $\Delta$'s). Among the two expectations of $S_1$-terms, one is the expectation of one term from $S_1$, and the other is the expectation of the product of either one or two terms from $S_1$. 
The first of these two expectations is bounded by $C(\lambda_2-\lambda_1)^{\kappa}$ by Lemma~\ref{lem:bounds_lower_moms}$(i)$ and $(iii)$. 
The second can be bounded by a constant using the triangular inequality and Lemma~\ref{lem:bounds_lower_moms}$(ii)$. 
Thus in total, with its prefactor, such a multinomial is bounded by $C(\lambda_2-\lambda_1)^{2\kappa}$, as claimed.
\item If $d_1 = 1$, then the corresponding multinomial is the product of the expectation of a product of one, two or three $S_1$-terms times the product of at least one expectation of the product of at most two $S_2$-terms. 
The expectations of $S_2$-terms can be bounded by constants using Lemma~\ref{lem:bounds_lower_moms}$(ii)$ (and the triangular inequality to remove the $\Delta$'s).
By Lemma~\ref{lem:bounds_lower_moms}$(i)$ and $(ii)$, the expectation of a product of one, two or three $S_1$-terms is bounded by $C(\lambda_2-\lambda_1)^{\kappa}$. Together with the prefactor, this bounds the multinomial by $C(\lambda_2-\lambda_1)^{2\kappa}$, as claimed.
\item Finally, if $d_1 = 0$, then the corresponding multinomial is a product of one, two, three or four expectations of $S_2$-terms. 
\begin{itemize}
\item If $d_2 = 1$, then the multinomial is one expectation of the product of four $S_2$-terms; among those four terms, two are from $\{\Delta W_n(\lambda_3), -\frac{\Delta t_3}{t_2}W_n(\lambda_2)\}$. In other words, the only possible multinomials are $\E(W_n(\lambda_2)^2 \Delta W_n(\lambda_3)^2)$, $-\frac{\Delta \lambda_3}{\lambda_2}\E(W_n(\lambda_2)^3 \Delta W_n(\lambda_3))$, and $\frac{(\Delta \lambda_3)^2}{\lambda_2}\E(W_n(\lambda_2)^4)$, all of which are bounded by $C(\lambda_2-\lambda_1)^{\kappa}$, by Lemma~\ref{lem:bounds_lower_moms}.
\item If $d_2 = 2$, then the multinomial is the product of two expectations of products of $S_2$-terms: these two products are either both products of two $S_2$-terms, or one of them has one term and the other three terms. In both cases, one can check that this multinomial with its prefactor can be bounded by $C(\lambda_2-\lambda_1)^{2\kappa}$.
\item If $d_2 = 3$, then the multinomial is the product of the expectation of the product of two $S_2$-terms times the product of two expectations of one $S_2$-term each. At least one these expectations contains a $\Delta$, and thus, by Lemma~\ref{lem:bounds_lower_moms}, it can be bounded by $C(\lambda_3-\lambda_1)^{\kappa}$, which, together with the prefactor, allows us to bound the monomial by  $C(\lambda_2-\lambda_1)^{2\kappa}$.
\item If $d_2 = 4$, then the multinomial is the product of four expectations of one $S_2$-term each. Two of these $S_2$-terms contain a $\Delta$ (because they are from a $T'$), and are thus bounded by $C(\lambda_3-\lambda_1)^{\kappa}$ (by Lemma~\ref{lem:bounds_lower_moms}$(i)$ and $(iii)$). The other two expectations are bounded by constants by Lemma~\ref{lem:bounds_lower_moms}$(ii)$. With the prefactor, we get $C(\lambda_2-\lambda_1)^{2\kappa}$, as claimed.
\end{itemize}
\end{itemize}

\paragraph{$T$-moments that involve no $S_3$ term and no $S_2$-terms.}
These $T$-moments only contain $S_1$-terms. 
These cases are a bit different from the previous ones because the prefactor $\Fac{x,0,0}$ is not small; however, by~{\HYPR} , it is bounded by a constant $C$ (since $\lambda_1\in[a,b]$). 
Since we want to bound every contribution (in absolute value) up to a constant, we can ignore the factorial moments here.
We look at the multinomials that come from the right-hand side of~\eqref{eq:exp_part}, and distinguish according to the value of $1\leq d_1\leq 4$ (note that $d_2 = d_3 = 0$ since the $T$-moment has no $S_2$ or $S_3$-terms).
\begin{itemize}
\item If $d_1 = 1$, then the corresponding multinomial is the expectation of the product of four $S_1$-terms. 
\begin{itemize}
\item If, among these four terms, at least two are from $\{T_2, T'_2\}$ (which contain $\Delta t_2$ and $\Delta t_3$ as factors, respectively), then, using Lemma~\ref{lem:Mom}$(ii)$ to bound the rest of the expectation by a constant, we get that this monomial is bounded by $C(\lambda_3-\lambda_1)^2$ and thus by $C(\lambda_3-\lambda_1)^{2\kappa}$ as claimed.
\item If exactly one of the four $S_1$-terms in the $T$-moment is from $\{T_2, T'_2\}$, then the multinomial is a constant times $\Delta t_3$ or $\Delta t_2$ times an expectation of the form
$\E(\Delta W_n(\lambda_2)^{k_1}\Delta W_n(\lambda_2)^{k_2})$, with $k_1 + k_2 = 3$. By Lemma~\ref{lem:Mom}$(iii)$ this last expectation is bounded by $C(\lambda_3-\lambda_1)^{\kappa}$, which, together with the $\Delta t_3$ or $\Delta t_2$ term gives $C(\lambda_3-\lambda_1)^{2\kappa}$ as claimed.
\item Finally, if none of the four terms in the $T$-moment are $T_2$ or $T'_2$, then the multinomial is $\Fac{1,0,0} \E(\Delta W_n(\lambda_2)^2\Delta W_n(\lambda_3)^2)$, which gives the first term in~\eqref{eq:final_rec}.
\end{itemize}
\item If $d_1\geq 2$, then the corresponding multinomial equals its prefactor times a product of at least two expectations of a product of one, two or three $S_1$-terms. By Lemma~\ref{lem:Mom}$(i)$ and $(iii)$, each of these two expectations is bounded by $C(\lambda_3-\lambda_1)^{\kappa}$ and thus their product is bounded by $C(\lambda_3-\lambda_1)^{2\kappa}$, as claimed.
\end{itemize}
This concludes the proof.

\subsubsection{Proof of Lemma \ref{lem:bounds_lower_moms}\it{(\ref{it:a1})}}\label{sub:proof_key_lemma1}

The result is immediate when $k_1+k_2+k_3=1$ or  $k_1+k_2+k_3=0$
since in this case $\ov{M}_{k_1,k_2,k_3}=1$. 
We reason by induction and assume that $\ov{M}_{k_1,k_2,k_3}<+\infty$ for all non-negative $k_1, k_2$ and $k_3$ such that $k_1+k_2+k_3\leq m$, where $m$ is some fixed integer in $\{1,2,3\}$. 
Take a triplet $(k_1, k_2, k_3)$ such that $k_1+k_2+k_3=m+1$; 
we only need to prove that $\ov{M}_{k_1,k_2,k_3}<+\infty$.
We set $M^{(n)}_{k_1,k_2,k_3}:=\E\l(W_n(\lambda_1)^{k_1}W_n(\lambda_2)^{k_2}W_n(\lambda_3)^{k_3}\r)$;
by \eref{eq:rel}, for all $n\geq 0$,
\begin{multline*}
M^{(n+1)}_{k_1,k_2,k_3}=\E\l[\prod_{\ell=1}^3 \l(\sum_{j=1}^{\ell} \frac{\S{j}{W_n(\lambda_\ell)}}{\lambda_\ell}\r)^{k_\ell}\r]\\
= \frac{1}{\prod_{\ell=1}^3 \lambda_\ell^{k_\ell}} \sum_{{k_{1,2}+k_{2,2}=k_2\atop k_{1,3}+k_{2,3}+k_{3,3}=k_3}}  \binom{k_1}{k_{1}}\binom{k_2}{k_{1,2},k_{2,2}}\binom{k_3}{k_{1,3},k_{2,3},k_{3,3}}\\
\nonumber \times\E\l[ \S1{W_n(\lambda_1)}^{k_1} \S1{W_n(\lambda_2)}^{k_{1,2}}\S2{W_n(\lambda_2)}^{k_{2,2}}\S1{W_n(\lambda_3)}^{k_{1,3}}\S2{W_n(\lambda_3)}^{k_{2,3}}\S3{W_n(\lambda_3)}^{k_{3,3}}\r].
\end{multline*}
We use \Cref{lem:Mom} to expand this expectation into a sum of multinomials.
Note that there are $k_1+k_{1,2}+k_{1,3}$ $S_1$-terms, $k_{2,2}+k_{2,3}$ $S_2$-terms and $k_{3,3}$ $S_3$-terms. Thus, the monomials appearing in the right-hand side of~\eqref{eq:exp_part} are at most of degree $m+1$.
The monomials with degree at most $m$ are uniformly bounded by the recurrence hypothesis. 
After expansion, we have a sum of multinomials (products of monomials) with total degree $m+1$. Bounding the monomials with degree at most~$m$ by a constant leaves us with a constant $c_0$ plus the contribution of monomials with degree~$m+1$.
Since $1/{\prod_{\ell=1}^3 \lambda_\ell^{k_\ell}}$ is also bounded on $[a,b]$ 
we have for $c= \max\{c_0/ \prod_{\ell=1}^3 \lambda_\ell^{k_\ell}, \lambda[3]\in[a,b]\}$,
\begin{multline*}
M^{(n+1)}_{k_1,k_2,k_3}\leq  c \\+\frac{\Fac{1,0,0} \E\l(\prod_{j=1}^3 W_n(\lambda_j)^{k_j}\r) +\1_{k_1=0}\Fac{0,1,0} \E\l(\prod_{j=2}^3 W_n(\lambda_j)^{k_j}\r)+\1_{k_1=k_2=0} \Fac{0,0,1} \E\l(W_n(\lambda_3)^{k_3}\r)}{\prod_{\ell=1}^3 \lambda_\ell^{k_\ell}}
\end{multline*}
and the reason for this is that the only terms with maximal degree comes from the cases where there are only terms of type $S_1$ or $S_2$ or $S_3$ (and for the second and third case, this can happen only if $k_1=0$ and $k_1=k_2=0$ respectively). Since $\Fac{1,0,0}=\lambda_1,\Fac{0,1,0}=\Delta \lambda_2,\Fac{0,0,1}= \Delta \lambda_3$, this gives 
\be
M^{(n+1)}_{k_1,k_2,k_3}&\leq & c +M^{(n)}_{k_1,k_2,k_3}\frac{\lambda_1 +\1_{k_1=0} \Delta \lambda_2+\1_{k_1=k_2=0} \Delta \lambda_3 }{\prod_{\ell=1}^3 \lambda_\ell^{k_\ell}}
\ee
Since $k_1+k_2+k_3\geq 2$, the factor of $M^{(n)}_{k_1,k_2,k_3}$ is uniformly bounded by $\leq 1/a$.
Now, to conclude, we use Lemma \ref{lem:rec} with $U_0=1$, $A=1/a$, $B=c$.

\subsubsection{Proof of Lemma \ref{lem:bounds_lower_moms}\it{(\ref{it:a4})}}
\label{sub:proof_key_lemma2}

First note that it is enough to prove the claim when $k_1+k_2=1$. 
Indeed, the case $k_1+k_2\geq 2$ can be reduced to the $k_1+k_2=1$ case by expanding $k_1+k_2-1$ factor of the type $\Delta W_n(\lambda_j)$ using the triangular inequality. 
For example
\[ |\E( (\Delta W_n(\lambda_2))W_n(\lambda_3)^2\Delta W_n(\lambda_3))|\leq|\E(  W_n(\lambda_1)W_n(\lambda_3)^2\Delta W_n(\lambda_3))|+|\E( W_n(\lambda_2)W_n(\lambda_3)^2\Delta W_n(\lambda_3))| \] 
is bounded from above by $2C (\lambda_3-\lambda_1)^{\kappa}$ 
if each term in the right-hand side is bounded by $C(\lambda_3-\lambda_1)^{\kappa}$.
If $k_1+k_2 = 1$, then either $k_1 = 1$ or $k_2=1$, and we need to treat these two cases separately.
We set 
\be
A_n &=& \E\l(\Delta W_n(\lambda_2) \prod_{i=1}^{D} W_n(\lambda_{m_i})\r)\\
B_n &=&  \E\l(\Delta W_n(\lambda_3) \prod_{i=1}^{D} W_n(\lambda_{m_i})\r)
\ee
where $1\leq D\leq 3$ (even if the method that follows work for larger $D$ when the corresponding moments exist), and the $m_i$ are, as usual, elements of $\{1,2,3\}$.
\paragraph{Control of $A_{n+1}$.} We want to prove that for any choice of $0\leq D\leq 3$, any choices of $(m_i)$, there exists a constant $C=C_{D,(m_i)}$ such that the corresponding sequence $(A_n)$ satisfies
\ben\label{eq:conclA}
\sup_{n\geq 0} |A_n|\leq C(\lambda_3-\lambda_1)^{\kappa} \textrm{ for all }a\leq \lambda_1\leq \lambda_2\leq \lambda_3\leq b.\een
We give a proof by recurrence on the value of $D$: if $D=0$ then $A_n=0$ so that \eref{eq:conclA} holds for $C=0$. Let us assume that we showed that $\sup_n|A_n|\leq C_{D',(m_i)}(\lambda_3-\lambda_1)^{\kappa}$ for all choices of $(D',(m_i))$ with $D'\leq D-1$ for some $D\in\{1,2,3\}$, 
and aim at proving the result for any $(D,(m_i))$. 
Fix such a pair $(D,(m_i))$.

We have by \eref{eq:theTi} and \eref{eq:rel}
\[A_{n+1}=\frac{1}{\lambda_2}\E\l[ \l(\S{1}{\Delta W_n(\lambda_2)- \frac{\Delta \lambda_2}{\lambda_1}W_n(\lambda_1)}+\S{2}{W_n(\lambda_2)}\r)\prod_{i=1}^D \sum_{\ell=1}^{m_i} \frac{\S{\ell}{W_n(\lambda_{m_i})}}{\lambda_{m_i}}\r].\]
We now use the linearity of $S_1$ and of the expectation and see $A_{n+1}$ as the sum of three expectations $A^{(1)}_{n+1}$, $A^{(2)}_{n+1}$ and $A_{n+1}^{(3)}$ that can be written as in the left-hand side of~\eqref{eq:exp_part}. 
We thus apply Lemma \ref{lem:Mom} to each of these three expectations, and get, from the right-hand side of~\eqref{eq:exp_part}, a sum of multinomials.
Recall that a multinomial is a prefactor $\Fac{}$ times a product of monomials.
The maximum degree of a monomial in the expansion of $A_{n+1}$ is $D+1$;
such monomials form a multinomial with their prefactor 
(i.e.\ they are not multiplied by another monomial).
In the expansion of $A^{(1)}_{n+1}$, the only monomial of degree $D+1$ comes from the partition that leaves all $S_1$-terms in one part. The same is true for $A^{(2)}_{n+1}$.
In $A^{(3)}_{n+1}$, we only get a multinomial of degree $D+1$ if $m_i\geq 2$ for all $1\leq i\leq D$, and it comes from the partition that leaves all $S_2$-terms in the same part. 
Thus, the only multinomials involving a monomial of degree $D+1$ are
\be
M_1 &=
& \frac{\Fac{1,0,0}}{\lambda_2\prod_{i=1}^D \lambda_{m_i}} \E\l(\Delta W_n(\lambda_2)\prod_{i=1}^D W_n(\lambda_{m_i})\r)
= \frac{\Fac{1,0,0}}{\lambda_2\prod_{i=1}^D \lambda_{m_i}} A_n,\\
M_2&=& -\frac{\Fac{1,0,0}}{\lambda_2\prod_{i=1}^D \lambda_{m_i}} \E\l(\l( \frac{\Delta \lambda_2}{\lambda_1}W_n(\lambda_1)\r) \prod_{i=1}^D W_n(\lambda_{m_i})\r),\\
M_3&=&\frac{\Fac{0,1,0}}{\lambda_2\prod_{i=1}^D \lambda_{m_i}} \E\l(\l(\Delta W_n(\lambda_2)- \frac{\Delta \lambda_2}{\lambda_1}W_n(\lambda_1)\r) \prod_{i=1}^D W_n(\lambda_{m_i})\r)\1_{m_i\geq 2, \forall 1\leq i \leq D}\\
   &=& \frac{\Fac{0,1,0}A_{n}}{\lambda_2\prod_{i=1}^D \lambda_{m_i}}\1_{m_i\geq 2, \forall 1\leq i \leq D}-\frac{\Fac{0,1,0}}{\lambda_2\prod_{i=1}^D \lambda_{m_i}} \E\l( \frac{\Delta \lambda_2}{\lambda_1}W_n(\lambda_1) \prod_{i=1}^D W_n(\lambda_{m_i})\r)\1_{m_i\geq 2, \forall 1\leq i \leq D}.
\ee 
Note that $|M_2|\leq C\Delta \lambda_2\leq C(\lambda_3-\lambda_1)^{\kappa}$ because of Lemma~\ref{lem:bounds_lower_moms}$(ii)$.
Hence, the total contribution of the monomial $A_n$ in $M_1$ and $M_3$ (and thus in the expansion of $A_{n+1}$) is
\[ \frac{\Fac{1,0,0}+\Fac{0,1,0}\1_{m_i\geq 2, \forall 1\leq i \leq D}}{\lambda_2\prod_{i=1}^D \lambda_{m_i}} \times A_n\]
All the other multinomials appearing in the expansion of $A_{n+1}$ (included the second term of $M_3$) satisfy one of the following alternatives:
\begin{itemize} 
\item Its prefactor is $\Fac{x,y,z}$ with $y\geq 1$ or $z\geq 1$ (meaning that, in the right-hand side of~\eqref{eq:exp_part}, it comes from a triplet $(d_1, d_2, d_3)$ such that $d_2\geq 1$ or $d_3\geq 1$). By \eref{eq:dsdity} for these values of $(x,y,z)$,$\Fac{x,y,z}\leq C(\lambda_3-\lambda_1)^{\kappa}$. 
Furthermore, all the monomials appearing in this multinomial can be bounded by constants by Lemma~\ref{lem:bounds_lower_moms}$(ii)$.
\item Its prefactor is $\Fac{x,0,0}$ (and $x\neq 1$ since this gives $M_1$): 
in this case, either $\Delta W_n(\lambda_2)$ appears in a monomial of degree at most $D-1$, 
or $\frac{(\Delta \lambda_2)}{\lambda_1} W_n(\lambda_1)$ appears in a monomial of degree at least one~$1$. 
Applying the induction hypothesis in the first case, and~Lemma~\ref{lem:bounds_lower_moms}$(i)$ in the second case (and Lemma~\ref{lem:bounds_lower_moms}$(ii)$ in both case to bound the other monomials involved in the multinomial), we get that, in absolute value, this multinomial is bounded by $C(\lambda_3-\lambda_1)^{\kappa}$.
\end{itemize}
We thus get that
\ben\label{rec:rec1}
A_{n+1}=\frac{\Fac{1,0,0}+\Fac{0,1,0}\1_{m_i\geq 2, \forall 1\leq i \leq D}}{\lambda_2\prod_{i=1}^D \lambda_{m_i}} A_n + \sum {\sf multinomials}\een
which gives by the triangular inequality
\ben\label{rec:rec2}
|A_{n+1}|\leq |A_n|/a + \sum |{\sf multinomials}|\een
and all multinomials in the sum are bounded in absolute value by $C(\lambda_3-\lambda_1)^{\kappa}$ for some $C>0$ (which can depend on the multinomial, but since there are finitely many of them, we can take the maximum constant for~$C$). 
The bound by $|A_n|/a$ comes from $|D|\geq 1$ and $\Fac{1,0,0}=\lambda_1$, $\Fac{0,1,0}=\lambda_2-\lambda_1$, and $a\leq \lambda_1 \leq \lambda_2\leq \lambda_3\leq b$.
Since $A_0=0$, we conclude by Lemma~\ref{lem:rec} that $|A_{n}|\leq C'(\lambda_3-\lambda_1)^{\kappa}$ for all $a\leq \lambda_1\leq \lambda_2\leq \lambda_3\leq b$.

\paragraph{Control of $B_{n+1}$.}  We apply the same strategy as for $A_{n+1}$: we reason by recurrence over $D$. Again the case $D=0$ is trivial since $B_n=0$ in this case. 
After that the formula are a bit more involved; 
let us have a glimpse on the differences with the $A_n$ case. 
We group a bit the $T'_i$ defined in \eref{eq:theTi} and write
\begin{multline*}B_{n+1}=\frac{1}{\lambda_3}\E\left[ \l(\S{1}{\Delta W_n(\lambda_3)- \frac{\Delta \lambda_3}{\lambda_2}W_n(\lambda_2)}+\S{2}{\Delta W_n(\lambda_3)- \frac{\Delta \lambda_3}{\lambda_2}W_n(\lambda_2)}+\S{3}{W_n(\lambda_3)}\r)\right.\\
  \times \prod_{i=1}^D \sum_{\ell=1}^{m_i} \frac{\S{\ell}{W_n(\lambda_{m_i})}}{\lambda_{m_i}}\left].{\phantom{\sum}}\right.
\end{multline*}
Again, notice the presence of $\Delta W_n(t_3)$ in a $S_1$ and a $S_2$ terms, 
while the $S_3$ terms concerns $W_n(\lambda_3)$. 
When one expands everything, and pack together the only terms -- those of maximum degree-- that contain $B_n$ as a factor, 
we observe that they can be produced only by $S_1$-terms, 
and possibly $S_2$-terms if all the $m_i\geq 2$.
We then get, for $M_1$, the contribution of these $B_n$ terms
\be
M_1&=&\l[\frac{\Fac{1,0,0}}{\lambda_3\prod_{i=1}^D \lambda_{m_i}}+\frac{\Fac{0,1,0}}{\lambda_3\prod_{i=1}^D \lambda_{m_i}} \1_{m_i\geq 2, 1\leq i \leq D}\r] \E\l(\l(\Delta W_n(\lambda_3)\r) \prod_{i=1}^D W_n(\lambda_{m_i})\r)
        \\
         &=& \l(\frac{\lambda_1}{\lambda_3\prod_{i=1}^D \lambda_{m_i}}+\frac{\Delta \lambda_2}{\lambda_3\prod_{i=1}^D \lambda_{m_i}}\1_{m_i\geq 2, 1\leq i \leq D}  \r) B_n.\ee
The rest of the terms coming from the expansion of $B_{n+1}$ involves either $\Delta W_n(\lambda_3)$, or $\Delta \lambda_3$, and the possible terms avoiding this contains a $S_3$ terms so that it comes with a prefactor $\Fac{x,y,z}$ with $z\geq 1$). This allows to write some equations similar to \eref{rec:rec1} and \eref{rec:rec2}:
\ben\label{eq:rehte}
B_{n+1}=\frac{\Fac{1,0,0}+\Fac{0,1,0}\1_{m_i\geq 2, \forall 1\leq i \leq D}}{\lambda_3\prod_{i=1}^D \lambda_{m_i}} B_n + \sum {\sf multinomials}\een
from which we conclude for the same reasons as in the $A_n$ case.

\section{Exact computations of the moments of $W_n$ and of $W$}
\label{sec:ECM}
In this short section, we would like to discuss the fact that 
the moments of $W_n$ and of $W$ can be computed (when they exist), 
and a closed formula for them can be derived. 
However, the formulae we obtain are so complicated that, despite important efforts, 
we were not able to find a way to present them in the paper: 
 some matrices with large size and with involved coefficients enter into play in the formula expressing the moments $\E(\Delta W_n(\lambda_2)^2\Delta W_n(\lambda_3)^2)$ in terms of the moments of $(X(\lambda_1), X(\lambda_2), X(\lambda_3))$. 
The obtained formulas are exact but we were unable to extract from them a simple criterion for the tightness.

We sketch the method allowing to get these close formulae: 
in principle, they can be used to treat some cases that are not covered by our Theorem~\ref{thm:Geninf} 
({\HYPM} was derived working with inequalities, 
and it probably does not cover all the cases for which $\E((\Delta W_n(\lambda_2))^2(\Delta W_n(\lambda_3))^2)\leq C(\lambda_3-\lambda_1)^{2\kappa}$). 
We focus on the 3-dimensional moments, but the same method applies for any higher-dimensional moments.

We just sketch the ideas:

{\bf (I) A non-linear recursion formula:} Using~\eqref{eq:rel}, we write 
\[M_{n+1}(k_1,k_2,k_3)=\E\bigg(\prod_{j=1}^3 W_{n+1}(\lambda_j)^{k_j}\bigg)
=\E\Bigg(\prod_{j=1}^3 \bigg(\sum_{\ell=1}^j \frac{\S{\ell}{W_n(\lambda_j)}}{\lambda_j}\bigg)^{k_j}\Bigg).\] By Lemma \ref{lem:Mom}, $M_{n+1}(k_1,k_2,k_3)$ can thus be written as $\l(\prod_{j=1}^3 \lambda_j^{-k_j}\r)\Fac{1,0,0}\times M_n(k_1,k_2,k_3)$ plus a sum of products of monomials of some $M_n(d_1,d_2,d_3)$ with $(d_1,d_2,d_3) < (k_1,k_2,k_3)$ (the inequality between vectors means non-strict inequality coordinate by coordinate and strictly smaller on at least one entry).

We thus get a recursive equation that gives $M_{n+1}(k_1, k_2, k_3)$ in terms of $M_n(k_1, k_2, k_3)$ and of lower order moments $M_n(d_1,d_2,d_3)$. This means that, in principle, one can calculate $M_n(k_1, k_2, k_3)$ recursively, for arbitrary $n$ and $(k_1, k_2, k_3)$. 
Unfortunately, the recursion formula is not linear in the lower order moments, which makes this computation more complex.

{\bf (II) Conservation of degrees:} When one uses Lemma~\ref{lem:Mom} to expand $M_{n+1}(k_1,k_2,k_3)$, the total degree in each multinomial on the right-hand side is $k_1+k_2+k_3$. 
Similarly, if one uses Lemma~\ref{lem:Mom} to expand, e.g., $M_{n+1}(d_1,d_2,d_3)\times M_{n+1}(d'_1,d'_2,d'_3)$ (applying Lemma~\ref{lem:Mom} to both terms), then, after expansion, the total degree of each multinomial appearing in the expansion is $d_1+d_2+d_3+d'_1+d'_2+d'_3$. In other words, the total degree of a multinomial is left unchanged by applying Lemma \ref{lem:Mom} to all its monomials. (This is true in all generality, even when multiplying more than two monomials.)

{\bf (III) Linearising the recursion formula:} 
A consequence of (I) and (II) is that it is possible to linearise 
the induction formula of (I). 
The idea is that, although the formula for $M_{n+1}(d_1, d_2, d_3)$ does not belong to the set
of linear combinations of the $M_n(d_1', d_2', d_3')$, with $(d'_1,d'_2,d'_3)\leq (d_1,d_2,d_3)$,
it belongs to the set of linear combinations of their products.
Furthermore, for a fixed value of $d_1+d_2+d_3$, there are finitely many of these products.
We thus take all these possible products (i.e.\ all monomial or product of monomials with total degree $d_1+d_2+d_3$) as a basis for this linear representation. 
(In fact, we can just sequentially add the products into the basis while running the computation to construct the smallest vector space that contains all necessary moments, and that is, somehow, stable by our rewriting rules.)

\medskip
Taking into account that $\E(W_n(\lambda_\ell))=1$ simplifies a bit the formulas:
some products of monomials of total degree~4 can be simplified. 
For example, $\E(W_n(\lambda_1)^3)\E(W_n(\lambda_1))=\E(W_n(\lambda_1)^3)$.

Applying (I-III) when calculating~$M_{n+1}(0, 2, 2)=\E((\Delta W_n(\lambda_2))^2(\Delta W_n(\lambda_3))^2)$,
we can write this monomial as a linear combination of products of monomials of total degree 4. Because of the simplifications due to $\E(W_n(\lambda_\ell))=1$, we sometimes see products of smaller total degree.
For example, some of the products appearing when writing $M_{n+1}(0, 2, 2)$ in term of $M_n(0,2,2)$ are, among others $M_n(1,3,0)$, $M_n(0,3,0)$, $M_{n}(1,1,0)M_n(0,1,1)$, and $M_n(0,2,0)^2$.
We give a name $P_n^{(i)}$ to each of the product of monomials that arises in this sum:
for example, set $P_n^{(1)}=M_n(1,2,1)$, $P_n^{(2)}=M_n(0,2,0)$, $ P_n^{(3)}=M_n(0,2,0)^2$, etc (we ignore the algebraic relation that can link these products of moments). 
In the end, one can construct a basis of $41$ of these products of monomials that allows to linearise the recursion of $M_n(0,2,2)$ as in (III). 
If one defines $V_n$ as the vector whose coordinates are the $P_n^{(j)}$, we eventually get that
\[V_n=AV_{n-1}+U,\]
for an explicit matrix~$A$ (whose coefficients are functions of the~$\Fac{d_1,d_2,d_3}$'s)
and a vector $U$ whose coordinates are the $P_0^{(j)}$.

The $41\times 41$ matrix $A$ can be diagonalised (in fact, up to relabelling the $P_n^{(i)}$'s, it is triangular). 
This provides some explicit formulae for all $P_n^{(i)}$'s by the standard mean of linear algebra.
These formulae are explicit but giant! several pages in standard A4 format are needed to write down their expression: at the end, of course, all moments of interests can be expressed in terms of the moments of $(X(\lambda_1),X(\lambda_2),X(\lambda_3))$.

The limiting moments~$P$ can also be computed: they are solution of
\[P=AP+U,\]
and since $A$ is diagonalisable, they can be exactly computed,
although again, the formula obtained doing this is huge and hard to manipulate.

\section{Remaining proofs}
\label{sec:RP}

\subsection{Proof of Lemma~\ref{lem:easy_HYPM}}\label{sec:proof_lem_simpl_cond}
Since $X$ is almost surely non-decreasing and integer-valued, one has
  \ben\label{eq:incr} X(\lambda_1)\leq X(\lambda_2)\leq X(\lambda_3).\een
$\bullet$  We start by proving (\eref{eq:dsdity} {\bf and }\eref{eq:dsdity2}) $\imp$  (\eref{eq:easy1}  {\bf and }\eref{eq:easy2}). 
If \eref{eq:dsdity2} holds, then 
\[
\E\l[(\Delta X(\lambda_2))^2(\Delta X(\lambda_3))^2\r]=
\sum_{y=1}^2\sum_{z=1}^2  \Fac{0,y,z}(\lambda_1,\lambda_2,\lambda_3)\leq 4(\lambda_3-\lambda_1)^{2\kappa}.\] So that \eref{eq:easy1} holds for $C'=4C$. Now, to prove that \eref{eq:easy2} holds, it suffices to express $\E(( \Delta X(\lambda_3)) (1+X_3^3))$ in terms of the factorial moments appearing in \eref{eq:dsdity} {\bf and }\eref{eq:dsdity2}, which is possible: 
\ben\label{eq:tdsfehze}
\E(( \Delta X(\lambda_3)) X(\lambda_3)^3)&=&
 \Fac{0, 0, 4} + \Fac{0, 0, 1}  + \Fac{3, 0, 1}+ \Fac{0, 3, 1}+ 7(\Fac{1, 0, 1}  + \Fac{0, 0, 2} + \Fac{0, 1, 1})\notag\\
 && + 12(\Fac{1, 1, 1} + \Fac{0, 1, 2}   + \Fac{1, 0, 2})+ 6(\Fac{1, 1, 2}  + \Fac{0, 0, 3} + \Fac{0, 2, 1}+ \Fac{2, 0, 1})\notag\\
 &&+ 3(\Fac{0, 1, 3} +\Fac{2, 1, 1}  + \Fac{1, 2, 1}  + \Fac{2, 0, 2} + \Fac{0, 2, 2}+ \Fac{1, 0, 3})
 \een
 {(This formula can be checked by hand; it follows from the fact that one can write $x_3^3(x_3-x_2)$ on the basis formed by $\prod_{i=1}^3 \prod_{j=0}^{n_i}(x_i-j)$, and it can be computed automatically, using a computer algebra system).}
 
$\bullet$ We now prove that  (\eref{eq:easy1}  {\bf and }\eref{eq:easy2}) $\imp$ (\eref{eq:dsdity} {\bf and }\eref{eq:dsdity2}). First, since $Y$ is integer-valued, we have $\E(Y^2)\geq \E(Y)$ and $\E(Y^2)\geq \E(Y(Y-1))$, and thus \eref{eq:easy1} implies \eref{eq:dsdity2} straightforwardly. 
Moreover, by~\eref{eq:incr},
\ben\label{eq:rhyrk}
C'(\lambda_3-\lambda_2)^{\kappa}\geq \E((\Delta X(\lambda_3)) X_3^3 \geq \E\l[(X_1^{j_1}X_2^{j_2}X_3^{j_3})\Delta X(\lambda_3)\r]\een for all $j_1,j_2,j_3$ such that $0\leq j_1+j_2+j_3\leq 3$ (note that when $j_1+j_2+j_3=0$, the right-hand side is zero).
Each element $\Fac{x,y,z}(\lambda_1,\lambda_2,\lambda_3)$ with $z\geq 1$ appearing in \eref{eq:dsdity} can be expanded as a sum of terms of the form $\E[(X_1^{j_1}X_2^{j_2}X_3^{j_3})\Delta X(\lambda_3)]$ (we write each $\Delta X(\lambda_3)$ and $\Delta X(\lambda_2)$ except one $\Delta X(\lambda_3)$ as a difference and then use of distributivity to expand). 
Therefore, \eref{eq:rhyrk} implies that $\Fac{x,y,z}(\lambda_1,\lambda_2,\lambda_3)\leq C(\lambda_3-\lambda_2)^\kappa$ for all $z\geq 1$. 
It only remains to treat the case $z=0$; 
in this case, we apply~\eref{eq:easy2} to $(\lambda_1, \lambda_2)$ instead of $(\lambda_2, \lambda_3)$ (this is allowed because $\lambda_3$ and $\lambda_2$ in~\eref{eq:easy2} are just any numbers satisfying $a\leq \lambda_2\leq \lambda_3\leq b$).
This gives $\E(( \Delta X(\lambda_2)) X(\lambda_2)^3)\leq C'(\lambda_2-\lambda_1)^\kappa$. 
From here, one can use the same arguments as in the case $z\geq 1$, to prove that~\eref{eq:dsdity} holds when $z=0$ and $y\geq 1$.

\subsection{Proof of \Cref{lem:unicite}}
Recall that the $L^2$ Wasserstein metric is defined as follows: 
for any two probability distributions $\mu$ and $\nu$ in $\mathcal M_2(1, \ldots, 1)$,
\[d_W(\mu, \nu)
= \inf\big\{\mathbb E\big[\|(U_1, \ldots, U_m)-(\hat U_1, \ldots, \hat U_m)\|_2^2\big]^{\nicefrac12} 
\colon (U_1,\ldots, U_m)\sim \mu, (\hat U_1, \ldots, \hat U_m)\sim \nu\big\}.\]
Note that if $\mathbb E[(U_1, \ldots, U_m)]= \mathbb E[(\hat U_1, \ldots, \hat U_m)]$, then
\[\mathbb E\big[\|(U_1, \ldots, U_m)-(\hat U_1, \ldots, \hat U_m)\|_2^2\big]
= \sum_{k=1}^m \mathrm{Var}(U_k-\hat U_k).
\]
Thus, for all $\mu, \nu \in \mathcal M_2(1, \ldots, 1)$, we have
\begin{equation}\label{eq:contr}
d_W\big(\Psi(\mu), \Psi(\nu)\big)^2
\leq\sum_{k=1}^m \mathrm{Var}\left(\frac1{\lambda_k}\sum_{i=1}^{X(\lambda_k)} (U_i^{\sss (i)}-\hat U_i^{\sss (i)})\right)
\end{equation}
for all $(U_1,\ldots, U_m)\sim \mu$ and $(\hat U_1, \ldots,  \hat U_m)\sim \nu$, 
where $((U_1^{\sss (i)}, \ldots, U_m^{\sss (i)}),(\hat U_1^{\sss (i)}, \ldots, \hat U_m^{\sss (i)}))_{i\geq 1}$ are sequences of i.i.d.\ copies of $((U_1, \ldots, U_m), (\hat U_1, \ldots, \hat U_m))$, independent from the offspring process $X$.
Using the law of total variance, we get that, for all $1\leq k\leq m$,
\begin{multline*}
\mathrm{Var}\left(\frac1{\lambda_k}\sum_{i=1}^{X(\lambda_k)} (U^{\sss (i)}_k-\hat U^{\sss (i)}_k)\right)\\
= \mathbb E\mathrm{Var}\left(\frac1{\lambda_k}\sum_{i=1}^{X(\lambda_k)} (U^{\sss (i)}_k-\hat U_k^{\sss (i)})\bigg| X(\lambda_k)\right)
+\mathrm{Var}\mathbb E\left[\frac1{\lambda_k}\sum_{i=1}^{X(\lambda_k)} (U^{\sss (i)}_k-\hat U^{\sss (i)}_k)\bigg| X(\lambda_k)\right]\\
= \frac1{\lambda_k^2}\mathbb E\big[X(\lambda_k)\mathrm{Var}(U_k-\hat U_k)\big]
= \frac{\mathrm{Var}(U_k-\hat U_k)}{\lambda_k}, \hspace{6.7cm}
\end{multline*}
where we have used again that $\mathbb E[U_k] = \mathbb E[\hat U_k]$,
and that $\mathbb EX(t) = t$ for all $t>1$.
Since the second term in Equation~\eqref{eq:contr} can be treated similarly, we get
\[d_W\big(\Psi(\mu), \Psi(\nu)\big)^2
\leq \sum_{k=1}^m \frac{\mathrm{Var}(U_k-\hat U_k)}{\lambda_k} 
\leq \frac{\mathbb E[\|(U_1, \ldots, U_m)-(\hat U_1, \ldots,\hat U_m)\|_2^2]}{\lambda_1}.
\]
Since this is true for all $(U_1, \ldots,U_m)\sim \mu$ and $(\hat U_1, \ldots, \hat U_m)\sim \nu$, taking the infimum gives
\[d_W\big(\Psi(\mu), \Psi(\nu)\big)\leq \frac1{\lambda_1} d_W\big(\mu,\nu\big),\]
which concludes the proof since $\lambda_1>1$.
 
\subsection{Proof of \Cref{pro:qdsqhsd}}

{As already mentioned, the subtrees of the root are themselves independent GW trees, and this leads us, notably to \eref{eq:grfdqdrz}, which says that $W_{n+1}(\lambda_i) = \frac1{\lambda_i}\sum_{i=1}^{X(\lambda_i)} W_n^{(i)}(\lambda_i),$ jointly for $1\leq i \leq d$. }  
Hence,
\be
W_n(\lambda_m)&=&\sum_{k=1}^m \sum_{j=1}^{\Delta X(\lambda_k)} \frac{W_{n-1}^{(j,k)}(\lambda_m)}{\lambda_m},\\
\Longrightarrow\sum_{m=1}^d x_m W_n(\lambda_m)
&=& \sum_{k=1}^d \sum_{j=1}^{\Delta X(\lambda_k)} \l[ \sum_{m=k}^d   x_m  \frac{W_{n-1}^{(j,k)}(\lambda_m)}{\lambda_m}\r]
\ee
The $(W_n^{(j,k)}(\lambda_m), 1\leq m \leq d)$ are independent and this is true also, conditionally on the $\Delta X(\lambda_\ell)$.
Hence taking in this last formula the operator $ \E(\exp({\rm i}~\cdot))$, 
the conclusion follows, as usual, by conditioning first by the values of $(\Delta X(\lambda_i), 1\leq i \leq d)$.\\

For the second statement, by \Cref{prp:qsdqs}, we know that the FDD of $W_n$ converges, so that $\Phi^{(n)}_{\lambda\cro{1,d}}$ converges simply, as $n \to+\infty$ to the Fourier transform $\Phi_{\lambda\cro{1,d}}$ of a $d$ dimensional distribution. Now, to conclude, it suffices to observe that  $x\cro{1,d}\mapsto \f^{\Delta}_{\lambda\cro{1,d}}(x\cro{1,d})$ is continuous on $\overline{B(0,1)^n}$ which contains the image set of the $\Phi^{(n)}_{\lambda\cro{1,d}}$.

\subsection{Proof of \Cref{lem:grfdq}}\label{sec:proof_cvproba}

The following proof is original even if we suspect it may exist elsewhere in the literature.

First, for $\Xi$ a $D[a,b]$ process, denote by ${\sf DP}(\Xi)$ the set of discontinuity points of $\Xi$, that is $t\in {\sf DP}(\Xi)$ if $\Xi(t)\neq \Xi(t^{-})$.
According to Billingsley \cite[p138]{billin}, $\P(t\in {\sf DP}(\Xi))>0$ is possible for at most countably many $t$. \par
As a consequence there exists a deterministic countable dense set $S$, such that  the set of continuity point of $\Xi$ contains $S$ almost surely. 

Under the hypothesis of the lemma, 
the statement holds in distribution (by Billingsley \cite[Section 12]{billin}): 
the sequence of processes $(T_n)$ converges in distribution in $D([a,b])$, 
and the FDD of the limit process $T'$ at its continuity points are determined, on a dense subset of it, 
by taking the limit of the FDD of $T_n$. 
To prove convergence in probability, we need more.

From the hypothesis, $T_n$ converges to $T$ on a dense subset of $[a,b]$.
 We want to prove that $\P( d(T_n,T')\geq \varepsilon)\sous{\longto}{n\to+\infty}0$ for any fixed $\varepsilon >0$, 
  where $T'$ is the c\`adl\`ag modification of $T$,
  \[d(f,g)=\inf_{\varpi} \max\{\|\varpi-{\sf Id}\|_{\infty}, \|f-g\circ \varpi\|_{\infty}\}\]
  where the infimum is taken on the set of strictly increasing and continuous functions $\varpi$ such that $\varpi(0)=0$ and $\varpi(1)=1$, and ${\sf Id}(y)=y$ on $[0,1]$. 
  Since the sequence $(T_n,n\geq 0)$ is tight in $D[a,b]$,
  for each $\varepsilon>0$,
\ben\label{eq:qsdqsd}
\lim_{\delta \to 0^+}\limsup_n \P(w'(\delta,T_n) \geq \varepsilon)=0
\een
where for a function $f:[a,b]\to \R$,
\ben \label{eq:wp}
w'(\delta,f)=\inf_{(\lambda_i)} \max_{i} w([\lambda_{i-1},\lambda_i),f)\een
and $w([c,d),f)=\sup\{ |f(x)-f(y)|, c\leq x,y <d\}$. The infimum in \eref{eq:wp} is taken on the set of lists $(\lambda_0,\cdots,\lambda_v)$ where $v$ is an integer, and the list satisfies: $\lambda_0=a, \lambda_v=b,$ and for each $i\in\{0,\cdots,v-1\}$, $\lambda_{i+1}-\lambda_i> \delta$ (this is called a $\delta$-sparse sequence).

The intervals $[\lambda_{i-1},\lambda_i)$ defined by the $(\lambda_i)$ will be called $(\lambda_i)$-intervals.

Choose a small $\varepsilon>0$ and then, a $\delta>0$ small enough, and $N_1$ large enough such that for any $n\geq N_1$
\ben\label{eq:dqdyq}
\P(w'(\delta,T_n) \geq \varepsilon) <\varepsilon ~~\textrm{ and }~~\P(w'(\delta,T) \geq \varepsilon) <\varepsilon.\een
This is possible by \eref{eq:qsdqsd}, and since $T$ is in $D[a,b]$. 
We may and will assume that
\ben\label{eq:dep}
\delta \leq \varepsilon.
\een

Now, take $(x_k, k \geq 0)$ a sequence in $[a,b]$ such that $\{x_k,k\geq 0\}$ is dense, and such that, the points of $\{x_k,k\geq 0\}$ are a.s.\ continuity points of $T$. Take a $K$ large enough, such that the connected components of $[a,b]\setminus\{x_0,\cdots,x_K\}$ have length $<\delta$; this is possible since $\{x_k,k\geq 0\}$ is dense.

Since $\l(T_n(x_1),\cdots,T_n(x_{K})\r)\as \l(T(x_1),\cdots,T(x_{K})\r)$, there exists $N_2$ such that for any $n\geq N_2$,
\ben\label{eq:gfeghet}
\max_{1\leq i\leq K} |T_n(x_i)-T(x_i)|\leq \varepsilon.\een
Take any (fixed) $n \geq \max\{N_1,N_2\}$; the event
\[E_{\varepsilon,n}=\{w'(\delta, T_n) \leq \varepsilon\}\cap \{w'(\delta, T) \leq \varepsilon\}\]
has probability at least $1-2\varepsilon$ by \eref{eq:dqdyq}.
When this event arises, there exists two $\delta$-sparse sequences $(\lambda_i)$ and $(t_i)$ such that
\[\max_i w([\lambda_{i-1},\lambda_i),T_n)\leq 2\varepsilon  \quad\textrm{ and } \quad \max_i w([t_{i-1},t_i),T)\leq 2\varepsilon.\]
Consider $(\widehat{x}_i,0\leq i\leq K)$ the list obtained by sorting 
increasingly the sequence  $(x_i,0\leq i\leq K)$. 
Since consecutive elements of $(\widehat{x}_i,0\leq i\leq K)$ are at most at distance $\delta$, when consecutive elements of the list $(\lambda_i)$ (resp. $(t_i)$) are at least at distance $\delta$, between two consecutive $\widehat{x}_i$ and $\widehat{x}_{i+1}$ can lie at most one element of $(\lambda_j)$, and at most one of $(t_j)$.

The main idea now is that $T_n$ (resp. $T$) may have big jumps of size $>\varepsilon$ at some of the elements of $(\lambda_j)$ (resp. $(t_j)$) but since $T_n$ and $T$ are close at the $(x_j)$ and have small variations between the $(\lambda_j)$ (resp. $(t_j)$), we can find a function $\varpi$ close to the identity to synchronize the big jumps. The details are as follows.

We define a suitable function $\varpi$ by working successively in each of the intervals $[\widehat{x}_i,\widehat{x}_{i+1}]$. Since the argument is the same in each interval, we choose an index $i\in\{0,\cdots,K-1\}$, we write $(x,x')$ instead of $(\widehat{x}_i,\widehat{x}_{i+1})$, and work in $[x,x']$. 
Three cases are possible:\\
(a) there is no element of $(\lambda_j)$ or of $(t_j)$ in $[x,x']$,\\
(b) there is a single element of $(\lambda_j)$ and a single element of $(t_j)$ in $[x,x']$,\\
(c) there is a single element of $(\lambda_j)$ in $[x, x']$ but none of $(t_j)$, or vice-versa.

\medskip
\noindent Case $(a)$: Both $x$ and $x'$ are in the same $(\lambda_j)$-interval $[\lambda_k,\lambda_{k+1})$ and in the same $(t_j)$ interval $[t_\ell,t_{\ell+1})$ (for some $k$ and $\ell$).
Hence, $w([x,x'], T_n)\leq w([\lambda_k,\lambda_{k+1}),T_n)\leq \varepsilon$.
Similarly, $w([x,x'], T)\leq \varepsilon$. 
Since $|T_n(x)-T(x)|\leq \varepsilon$ by \eref{eq:gfeghet}, we get $\sup_{y\in[x,x']}|T_n(y)-T(\lambda(y))|=\sup|T_n(y)-T(\lambda(y))|\leq 3\varepsilon$.

\medskip
\noindent Case $(b)$: Let~$\lambda$ denote the element of the list $(\lambda_j)$ lying in $[x,x']$, and by $t$ the element of $(t_j)$ lying in $[x,x']$. 
Also let $\lambda^p$ and $\lambda^f$ denote the elements of $(\lambda_j)$ preceding and following $\lambda$, 
and $t^{p}$ and $t^{f}$ denote  the element preceding and following $t$ in $(t_j)$.

The jump of $T_n$ at $\lambda$ and the one of $T$ at $t$ can be huge compared to $\varepsilon$, but they are almost equal.
Indeed, since $|T_n(x)-T(x)|\leq \varepsilon$ 
and $|T_n(x')-T(x')|\leq \varepsilon$, 
before both jumps and after both jumps, 
the two processes $T_n$ and $T$ are close to each other. 
More precisely, $w([\lambda ,x'],T_n)\leq w([\lambda,\lambda^f),T_n)\leq \varepsilon$,  
$w([t,x'],T)\leq \varepsilon$, $w([x,\lambda),T_n)\leq \varepsilon$, and $w([x,t),T)\leq \varepsilon$. 
This implies that
\[ |T_n(\lambda)-T(t)|\leq 4\varepsilon, |T_n(\lambda^{-})-T(t^-)|\leq 4\varepsilon\]
where left limit is denoted by the ``minus exponent'', which implies that
\[ |(T_n(\lambda)-T_n(\lambda^-))-(T(t)-T(t^-)|\leq 8\varepsilon.\]
We need to use $\varpi$ to synchronize the jump:
take $\varpi$ as the linear function by part that sends\\
-- $[x,t]$ linearly onto $[x,\lambda]$, and \\
--  $[t,x']$ linearly onto $[\lambda,x]$.\\
Globally, since $|x-x'|\leq \delta$, $|\varpi(y)-y|\leq \delta$.\par
We have on $[x,x']$, 
\ben\label{eq:hgryvs}
\max_{y\in[x,x']} \l|T_n(y)-T(\varpi(y))\r|\leq 10\varepsilon.\een

\medskip
\noindent Case $(c)$: we take again $\varpi(y)=y$ on $[x,x']$. 
By symmetry assume that there is an element $\lambda$ of $(\lambda_j)$ 
in $[x,x']$ but none of $(t_j)$. 
In this case, by a similar argument to Case (b), one can see that the jump of $T_n$ at $\lambda$ must be smaller than $4\varepsilon$, and the conclusion follows. 
We thus have $\max_{y\in[x,x']} |T(y)-T(x)|\leq 6\varepsilon$.

\medskip
In total, we showed that on $[x,x']$
\[\max\{|\varpi(u)-u|, u \in [x,x']\} \vee \max\{|T_n(y)-T(\varpi(y))|,y\in[x,x']\}\leq \delta+10\varepsilon\]
so that by \eref{eq:dep}, this is smaller than 11$\varepsilon$. This implies that $d(T_n,T)\leq 10\varepsilon$ with probability at least $1-2\varepsilon$.

\newpage

\tableofcontents

\end{document}

%% file: preambule.tex
\usepackage[normalem]{ulem}
\usepackage[english]{babel}
\usepackage[utf8]{inputenc}
\usepackage[T1]{fontenc}
\usepackage{lmodern}
\usepackage{graphicx}
\usepackage{amsfonts}
\usepackage[nice]{nicefrac}
\usepackage{amsmath,amsthm,amssymb} %,mathabx}
\usepackage{thmtools}
\usepackage{caption}
\usepackage{array,cleveref}
\usepackage{color}
\usepackage{url}
\usepackage{paralist}
\usepackage[top=3cm,bottom=3.5cm,right=3cm,left=2cm]{geometry}
\usepackage{enumerate}
\usepackage{subfig}
\usepackage[affil-it]{authblk}

\newcommand{\sss}{\scriptscriptstyle}

\def \S#1#2{S_{#1}\l(#2\r)}

\newcommand{\compact}{ \topsep0pt   \itemsep=0pt   \partopsep=0pt   \parsep=0pt}

\newcounter{c}
\def \bir{\begin{itemize}\compact \setcounter{c}{0}}
\def \itr{\addtocounter{c}{1}\item[($\roman{c}$)]} %i ii
\def \eir{\end{itemize}\vspace{-2em}~}

\newcounter{d}
\def \bia{\begin{itemize}\compact \setcounter{d}{0}}
\def \eia{\end{itemize}\vspace{-2em}~}
\def \longto{\longrightarrow}

\def \f{{\sf f}}

\newcommand \Fac[1]{{\sf Fac}_{#1}^{\Delta}}

\def \uniform{{\sf Uniform}}
\def \Bernoulli{{\sf Bernoulli}}

\def \cro#1{{\llbracket #1\rrbracket}}
\newcommand{\Bin}{{\sf Bin}}
\newcommand{\Geom}{{\sf Geo}}
\newcommand{\Poi}{{\sf Poi}}

\def \ba{\begin{align}}
\def \ea{\end{align}}
\def \be{\begin{eqnarray*}}
\def \ee{\end{eqnarray*}}
\def \ben{\begin{eqnarray}}
\def \een{\end{eqnarray}}

\def \beq{\begin{equation}}
\def \eq{\end{equation}}
\def \eref#1{(\ref{#1})}

\def \l{\left}
\def \r{\right}
\def \sss{}
\def \as{\xrightarrow[n]{\sss (as.)}}
\def \dd{\xrightarrow[n]{(d)}}

\newcommand{\bs}{\boldsymbol}

\newcommand{\C}{\mathbb{C}}
\newcommand{\R}{\mathbb{R}}

\def \ov#1{\overline{#1}}

\newcommand{\E}{\mathbb{E}}
\renewcommand{\P}{\mathbb{P}}

\newcounter{a}

\newcommand{\Ber}{{\sf Ber}}

\newtheorem{thm}{Theorem}[section]
\newtheorem{prop}[thm]{Proposition}
\newtheorem{df}[thm]{Definition}
\newtheorem{lem}[thm]{Lemma}

\newtheorem{rem}[thm]{Remark}

\newtheorem{OQ}{Open question}

\def \sur#1#2{\mathrel{\mathop{\kern 0pt#1}\limits^{#2}}}
\def \sous#1#2{\mathrel{\mathop{\kern 0pt#1}\limits_{#2}}}

\def \eqd{\sur{=}{(d)}}
\def \dd{\xrightarrow[n]{(d)}}

\makeatletter 
\renewenvironment{proof}[1][Proof] {\par\pushQED{\qed}\normalfont\topsep6\p@\@plus6\p@\relax\trivlist\item[\hskip\labelsep\bfseries#1\@addpunct{.}]\ignorespaces}{\popQED\endtrivlist\@endpefalse} 
\makeatother